\documentclass[11pt]{article}
\pdfoutput=1

\usepackage[utf8]{inputenc}
\usepackage[T1]{fontenc}
\usepackage[letterpaper,margin=1in]{geometry}
\usepackage{microtype}
\usepackage[usenames,dvipsnames]{xcolor}
\usepackage{amsmath,amssymb,amsthm,mathtools}
\usepackage{enumitem}
\usepackage{booktabs,array,needspace,longtable,calc}
\usepackage[authoryear,round]{natbib}
\usepackage{hyperref}
\usepackage[nameinlink,noabbrev]{cleveref}
\crefname{theorem}{Theorem}{Theorems}
\crefname{lemma}{Lemma}{Lemmas}
\crefname{corollary}{Corollary}{Corollaries}
\crefname{proposition}{Proposition}{Propositions}
\crefname{definition}{Definition}{Definitions}
\crefname{section}{Section}{Sections}
\crefname{subsection}{Section}{Sections}
\crefname{subsubsection}{Section}{Sections}
\crefname{appendix}{Appendix}{Appendices}
\crefname{table}{Table}{Tables}
\usepackage{etoc}  

\hypersetup{
  pdftitle={Boolean Small-Ball Inequalities for Discrepancy Theory},
  pdfauthor={Emrullah Akbas, Suvrit Sra},
  pdfsubject={Boolean determinant partitions, entropy, and matrix discrepancy},
  pdfkeywords={Boolean small-ball inequality, matrix discrepancy, Fisher information, Kadison-Singer, Matrix Spencer}
}

\usepackage{ss}

\definecolor{cdarkblue}{RGB}{30,75,170}
\definecolor{cdarkred}{RGB}{180,0,0}
\definecolor{cdarkgreen}{RGB}{0,130,0}
\newcommand{\darkblue}[1]{{\color{cdarkblue} #1}}
\newcommand{\darkred}[1]{{\color{cdarkred} #1}}

\setlist[itemize]{topsep=0.35em,itemsep=0.2em,leftmargin=2em}
\setlist[enumerate]{topsep=0.35em,itemsep=0.2em,leftmargin=2.2em}
\allowdisplaybreaks[1]

\newcommand{\R}{\mathbb R}
\newcommand{\C}{\mathbb C}
\newcommand{\E}{\mathbb E}
\newcommand{\Prob}{\mathbb P}
\newcommand{\dd}{\mathrm d}
\newcommand{\Fro}{\mathrm F}
\newcommand{\bstar}{\beta_\ast}
\newcommand{\nlsum}{\sum\nolimits}
\DeclareMathOperator{\Tr}{tr}
\DeclareMathOperator{\Sym}{Sym}
\DeclareMathOperator{\Herm}{Herm}
\DeclareMathOperator{\diag}{diag}
\DeclareMathOperator{\Law}{Law}
\DeclareMathOperator{\vecop}{vec}

\DeclareMathOperator{\disc}{disc}

\newtheorem{theorem}{Theorem}[section]
\newtheorem{proposition}[theorem]{Proposition}
\newtheorem{lemma}[theorem]{Lemma}
\newtheorem{corollary}[theorem]{Corollary}
\theoremstyle{definition}
\newtheorem{definition}[theorem]{Definition}
\theoremstyle{remark}

\numberwithin{equation}{section}

\title{Boolean Small-Ball Inequalities for Discrepancy Theory}
\author{%
\name Emrullah Akbas \email{emrullah.akbas@tum.de}\\
\addr Technical University of Munich, Garching, Germany\\[2pt]
\name Suvrit Sra \email{s.sra@tum.de}\\
\addr Technical University of Munich, Garching, Germany\\[2pt]
}

\begin{document}
\maketitle

\begin{abstract}
We prove new small-ball inequalities for boolean matrix-series. The leading example is $\E_s\bigl[{\det(I-S^2)^\beta\,\mathbf 1_{\{\|S\|<1\}}}\bigr]\ge e^{-O(\beta\tau)}$, which holds for \emph{boolean matrix-series} $S=\sum_i s_iA_i$ formed using symmetric matrices $A_1,\dots,A_n$ and uniformly random signs $s\in\{\pm1\}^n$. Specifically, this inequality holds for all $\beta\ge1$ with $\tau=\sum_i\Tr A_i^2$, as soon as the maximum of $(\Tr A_i^2)_{i=1}^n$ and a certain variance term are bounded above by universal constants. The proof combines the Gaussian reciprocal estimate of Akbas and Sra~\citeyearpar{AS26}, the directional-variation signing theorem of Guo, Fang, and Lu~\citeyearpar{GFL26}, and a replica argument that turns existence into a
Gibbs law on good signings. Most notably, boolean small-ball delivers a \darkblue{\emph{new, interlacing-free proof of
Kadison--Singer}} (most general case); it also recovers \darkblue{\emph{Matrix Spencer and Koml\'os as quick
corollaries}}, while yielding more than six almost immediate proofs of an
assortment of discrepancy theoretic problems.\end{abstract}



\section{Introduction}\label{sec:introduction}

Discrepancy problems seek a coloring of items into two nearly balanced parts. A recurring phenomenon is the gap between the best coloring and a typical one: Spencer's six standard deviations~\citep{Spencer85} beats the random split by a logarithmic factor. Since then discrepancy theory has been shaped by a range of methods, from partial coloring and entropy~\citep{Spencer85,LovettMeka2015}, to convex-geometric approaches~\citep{Ban98,Gluskin1989Extremal,Giannopoulos1997VectorBalancing}, to spectral and polynomial methods~\citep{MSS15,KyngLuhSong2020FourDeviations}.

We pursue a more unified perspective by identifying a common theme underlying many (matrix) discrepancy problems: a \emph{boolean small-ball} (BSB) estimate which lower-bounds the probability that a uniformly random signing lands in a small (but usable) set of good ones. The resulting estimate then yields several fundamental results in discrepancy theory as short corollaries.

\subsection{Summary of the main result}
We study discrepancy problems involving $n$ real symmetric $d\times d$ matrices $(A_1,\ldots,A_n)$. To state our results, we need to define the following notion of ``\emph{variance}'': $\nu:=\min_{U\in O(d)}\|\nlsum_i(A_i-P_UA_i)^2\|$, where $P_U$ denotes orthogonal projection onto the subspace of matrices that are diagonal in the basis $U$. 
The variance parameter $\nu$ measures how close the family is to being commutative; e.g., $\nu=0$ whenever the matrices $A_i$ are simultaneously diagonalizable; and also that $\nu \le4\| \sum_{i=1}^nA_i^2\|$, which puts it  within a factor 4 of the usual variance. Our main result is the following theorem.

\begin{theorem}[Boolean matrix small-ball]\label{thm:main}
There exist universal constants $\nu_0$, $q_1$, and $C_1$ for which the following holds. Let $A_1,\dots,A_n\in\Sym_d(\R)$ and $r>0$ satisfy
the variance bound $\nu\le \nu_0 \times r^2$ and Frobenius square-norm bound $\max_i\Tr A_i^2\le q_1 \times r^2$; also let $S_x=\sum_{i=1}^n x_iA_i$ with $x$ uniform on
$\{-1,1\}^n$. Then, using the shorthand $\tau :=\nlsum_i\Tr A_i^2$, for every $\beta>0$ we have the (det-scaled) small-ball inequality
\begin{equation}\label{eq:brsc-scaled}
\E_{x\sim\{\pm1\}^n}\!\Bigl[
\det\!\bigl(I- r^{-2}S_x^2\bigr)^\beta
\mathbf{1}_{\{\|S_x\|<r\}}
\Bigr]
\ge
\exp\!\bigl(-r^{-2}C_1\max(1,\beta)\tau\bigr).
\end{equation}
Moreover, there is a probability law $\mu_\beta$, supported on $\{x\in\{-1,1\}^n:\|S_x\|<r\}$
such that
\begin{equation}\label{eq:entropy-main}
D(\mu_\beta\|U_n)
+
\beta\,\E_{\mu_\beta}
\Phi\!\left(\tfrac{S_x}{r}\right)
\le
C_1\max(1,\beta)\frac{\tau}{r^2},
\end{equation}
where $\Phi(H):=-\log\det(I-H^2)$ ($:= +\infty$ if $\|H\|\ge1$) and $U_n$ is the uniform distribution on $\{\pm1\}^n$.
\end{theorem}

Letting \(\beta\downarrow0\) in \eqref{eq:brsc-scaled} gives the
ordinary (unscaled) BSB inequality,
\begin{equation}\label{eq:ordinary-bsb}
\Prob_{x\sim U_n}\bigl(\|S_x\|<r\bigr)
\ge
\exp\biggl(-C_1\frac{\tau}{r^2}\biggr).
\end{equation}
Indeed, for every \(x\) with \(\|S_x\|<r\),
$\det\bigl(I-r^{-2}S_x^2\bigr)>0,$
and therefore $\det\bigl(I-r^{-2}S_x^2\bigr)^\beta
\longrightarrow1$ as $\beta\downarrow0$. Moreover, for \(x\) with \(\|S_x\|\ge r\)  the indicator
\(\mathbf 1_{\{\|S_x\|<r\}}\) vanishes. Hence, the integrand in
\eqref{eq:brsc-scaled} converges pointwise to
to \(\mathbf 1_{\{\|S_x\|<r\}}\). Since the expectation is over the
finite set \(\{-1,1\}^n\), the limit may be interchanged with the expectation.\\

In particular, from the BSB inequality \eqref{eq:ordinary-bsb} we obtain that here are at least $\#\{x:\|S_x\|<r\}\ge2^ne^{-C_1\frac\tau{r^2}}$ colorings with discrepancy $r$.

\vskip8pt
Theorem~\ref{thm:main} relies on two building blocks: a core lemma in our work on Matrix Spencer~\citep[Proposition~3.8]{AS26} and a key step in the recent breakthrough on the Koml\'os conjecture~\citep[Corollary~2.3]{GFL26}.\footnote{Indeed, while we had identified the BSB lemma as a core result worth proving, our attempts did not fully materialize unconditionally until the Koml\'os ingredient became available.} We will explain below how these ideas fit together.

\vspace*{-5pt}
\subsection{Applications of the BSB Theorem}
\vspace*{-5pt}
The BSB theorem delivers a \textbf{\emph{new, interlacing-free proof
of Kadison--Singer}}; it also recovers \emph{Matrix Spencer and Koml\'os as quick corollaries}---see Section~\ref{sec:applications}; it also yields almost immediate proofs of numerous other problems in discrepancy theory, which we relegate for now to Appendix~\ref{app:appl}  to avoid clutter. Each entry of Table~\ref{tab:implications} is obtained from Theorem~\ref{thm:main} by choosing the
coefficients $B_i$, often as block direct sums, and reading off the
determinant, counting, and entropy statements. 
\begin{table}[t]\footnotesize\centering
\caption{Applications of Theorem~\ref{thm:main}; all are proved in Section~\ref{sec:applications}.}\label{tab:implications}
\renewcommand{\arraystretch}{1.1}
\begin{tabular}{@{}>{\raggedright\arraybackslash}p{7.4cm}>{\raggedright\arraybackslash}p{8.4cm}@{}}
\toprule
Result & Note\\
\midrule
Kadison--Singer, partition scale $O(\sqrt\varepsilon)$ (Cor.~\ref{cor:ks}) &
First proof without interlacing polynomials; comes with a count and an
entropy law for the good partitions.\\
\addlinespace
Matrix Spencer, signings of norm $O(\sqrt n)$ (Cor.~\ref{cor:ms}) &
Exponentially many such signings; uses only the Gaussian
estimate of \citet{AS26}, not its conclusion.\\
\addlinespace
Koml\'os and matrix Koml\'os, $\disc(B)\le K\sqrt{q+\nu(B)}$
(Cors.~\ref{cor:komlos}, \ref{cor:matrix-komlos}) &
Second parameter necessary (Prop.~\ref{prop:star-komlos}); relies on
the signing theorem of \citet{GFL26}.\\
\addlinespace
Unbiased rounding with exact marginals, error $O(\sqrt{q+\nu})$
(\S\ref{linear-discrepancy-with-prescribed-marginals}) &
Explicit relative-entropy budget to the product law.\\
\addlinespace
Multicolor discrepancy and exact quotas, error $O(\sqrt{2q+\nu})$
(\S\ref{arbitrary-multicolor-probabilities}--\S\ref{exact-quotas-and-their-boundary}) &
Error independent of the number of colors.\\
\addlinespace
Fully dynamic discrepancy, error $O(\sqrt{q_t+\nu_t})$
(\S\ref{adaptive-insertions-and-deletions-with-recourse}) &
Adaptive insertions and deletions, logarithmic recourse.\\
\addlinespace
Prefix discrepancy, bound $O(\sqrt\nu+\sqrt q\log n)$
(\S\ref{selectable-order-with-a-contracting-variance}--\S\ref{prescribed-order-and-interval-constraints}) &
Selectable order via a variance-contracting lift; joint dyadic law for
prescribed orders.\\

\addlinespace
Sparse systems and common-sign budgets
(\S\ref{a-general-common-sign-budget}--\S\ref{ks-and-arbitrary-scalar-side-constraints-together}) &
One signing for arbitrary scalar and matrix constraints at once,
including Kadison--Singer with side constraints.\\
\addlinespace
Graph signing and spectral bisection, $\|A_s\|=O(\sqrt\Delta)$
(\S\ref{graph-adjacency-and-signed-degrees}--\S\ref{spectral-edge-bisection-with-additive-degree-control}) &
$O(1)$ signed degrees; half-selections approximating half the
Laplacian with additive degree error.\\
\bottomrule
\end{tabular}
\end{table}
The appendix also records exact counterexamples that show  limitations of the
method BSB. 

\vspace*{-5pt}
\subsection{Related work}
\vspace*{-3pt}
Spencer's theorem~\citep{Spencer85} established constant multiples of $\sqrt n$ in the scalar square discrepancy problem. Banaszczyk~\citeyearpar{Ban98} connected vector balancing with convex geometry
and Gaussian measure. Smirnov and Vershynin~\citeyearpar{SV26} developed a
Fisher-information approach controlling discarded steps in an online
random walk; the directional-variation theorem of
Guo--Fang--Lu~\citeyearpar{GFL26} supplies the full-signing input used here.
On the matrix side, Akbas and Sra~\citeyearpar{AS26} developed Gaussian
determinant and reciprocal estimates in their proof of Matrix Spencer.

Our focus is the determinant partition.
The discrepancy inequality $\disc(B)\le C_E\sqrt{v+q}$, where
$\disc(B):=\min_{s\in\{-1,1\}^n}\|S_s\|$, appears as an intermediate theorem; it is also equivalent to BSB up to
universal constants. In Section~\ref{sec:applications}, BSB gives Matrix
Spencer and trace-controlled positive semidefinite partitions; its
off-diagonal refinement gives Koml\'os and the matrix extension above.
In the
rank-one case, the latter give the Kadison--Singer partition scale
$O(\sqrt\varepsilon)$, in the discrepancy formulation of
Weaver~\citeyearpar{Weaver04} resolved by Marcus, Spielman, and
Srivastava~\citeyearpar{MSS15}. The applications recover universal scales,
without asserting the sharp constants from that work.
Likewise, the Koml\'os application relies on the general geometric
theorem of \citet{GFL26}, rather than assuming the cube-signing
conclusion. It is a corollary of the stronger master BSB theorem, but is not an independent proof of the source's Koml\'os result.

\vskip6pt
\noindent\textbf{Note (added: 16 Sep' 26)}: While completing our paper, we noticed the remarkable work of~\citet{EJ26} that provides a deterministic ``interlacing free'' constructive algorithm for rank-1 Kadison--Singer. Their rank-1 algorithm actually admits a quick generalization to the general-rank Kadison--Singer problem; we note these details and also simple statements about the hardness of approximation beyond a factor of $2$ in Appendix~\ref{app:ks}.

\vspace*{-8pt}
\section{Preliminaries} \label{sec:preliminaries}
We start with notation conventions, key definitions, and some basic but useful lemmas. 
We use \(\|\cdot\|\) for the operator norm. Moreover, we write
\(\bstar:=\max(1,\beta)\). We denote the uniform distribution  on $\{-1,1\}^n$ by $U_n$. For probability distributions \(\mu\) and \(\nu\)
on a finite set, we write $D(\mu\|\nu)
=
\sum_x \mu(x)\log\frac{\mu(x)}{\nu(x)}$
for their relative entropy, with the usual convention that
\(D(\mu\|\nu)=+\infty\) if \(\mu\) is not absolutely continuous
with respect to \(\nu\). The following \emph{spectral log-barrier} and the corresponding Gibbs measure will be central to our arguments:
\begin{equation}\label{eq:barrier}
 \Phi(H)=
 \begin{cases}
 -\log\det(I-H^2),&\|H\|<1,\\
 +\infty,&\|H\|\ge1,
 \end{cases} \qquad
 W_\beta(H)=e^{-\beta\Phi(H)},\quad \beta>0.
\end{equation}
In particular, $W_\beta$ is zero on the spectral boundary as well as
outside the ball. Moreover, we denote the associated partition function by, $Z_\beta(H_x)=\E_{x\sim U_n}W_\beta(H_x).$ We start with an elementary lemma that notes some useful properties of the \emph{Gibbs weight} $W_\beta$. 
\begin{lemma}\label{lem:determinant}
For every $\beta>0$, $W_\beta$ is continuous and takes values in
$[0,1]$. Moreover,
\begin{equation}\label{eq:elementary-weight}
 W_\beta(H)\ge (1-\Tr H^2)_+^\beta.
\end{equation}
If $\|H\|\le r<1$, then
\begin{equation}\label{eq:energy-barrier}
 \Tr H^2\le \Phi(H)\le\frac{\Tr H^2}{1-r^2}.
\end{equation}
\end{lemma}
\begin{proof}
Continuity of \(W_\beta\) on \(\Sym_d(\mathbb R)\) is elementary, and since $\Phi(H) \ge 0$, it follows that $W_\beta(H) \in [0,1]$. We now prove~\eqref{eq:elementary-weight}. Suppose first that
\(\sum_{j=1}^d \lambda_j^2=\Tr H^2<1\).  Set \(a_j:=\lambda_j^2\). 
Using the elementary inequality 
 $\prod_{j=1}^d(1-a_j)\ge1-\sum_{j=1}^d a_j$, we obtain $\det(I-H^2)=\prod\nolimits_{j=1}^d(1-\lambda_j^2)\ge1-\nlsum_{j=1}^d\lambda_j^2=1-\Tr H^2$, and therefore also $W_\beta(H)=\det(I-H^2)^\beta\ge(1-\Tr H^2)^\beta$. For  \(\Tr H^2\ge1\), we have $(1-\Tr H^2)_+^\beta=0,$ so \eqref{eq:elementary-weight} follows in
both cases. Next, for~\eqref{eq:energy-barrier}, suppose that \(\|H\|\le r<1\). Then \(\lambda_j^2\le r^2\) for every \(j\). Using 
the elementary inequality $u
\le
-\log(1-u)
\le
\frac{u}{1-r^2}$ with \(u=\lambda_j^2\) and summing over \(j\) inequality~\eqref{eq:energy-barrier} is immediate. 
\end{proof}

Next, using Lemma~\ref{lem:determinant} we show that for matrix series with small variance, a simple argument suffices to obtain lower bounds on the  associated partition function. 
\begin{lemma}\label{lem:small-trace}
Let $A_1,\dots,A_n\in\Sym_d(\R)$. Let $w_i = \Tr(A_i^2)$ and write $\tau= \sum_{i=1}^nw_i$. Suppose that $0\le\tau\le1/2$. Then we have ,
$Z_\beta = \E_{x\sim U_n}W_\beta(\sum_{i=1}^nx_iA_i) \ge e^{-2\bstar\tau}$.
\end{lemma}

\begin{proof}
Let $S_x=\sum_{i=1}^n x_iA_i$ and note that $\E_{x\sim U_n}\Tr S_x^2=\tau$.
Suppose first that \(\beta\ge1\). By Lemma~\ref{lem:determinant}, $W_\beta(S_x)
\ge
\bigl(1-\Tr S_x^2\bigr)_+^\beta$.
The function $f(u)=(1-u)_+^\beta$
is convex for \(\beta\ge1\), so by Jensen we get
\[
Z_\beta = \E W_\beta(S_x) \ge \E\bigl(1-\Tr S_x^2\bigr)_+^\beta \ge \left(1-\E\Tr S_x^2\right)_+^\beta= (1-\tau)^\beta,
\]
where in the last step we used \(\tau\le1/2<1\). It remains to bound \(1-\tau\) from below. For
\(0\le\tau\le1/2\), $-\log(1-\tau)
=
\int_0^\tau\frac{du}{1-u}
\le
2\tau,$
and therefore $1-\tau\ge e^{-2\tau}.$
Hence, $Z_\beta
\ge
(1-\tau)^\beta
\ge
e^{-2\beta\tau}$.

Finally, let \(0<\beta<1\). Since \(\Phi\ge0\), we have that $W_\beta(H)
=
e^{-\beta\Phi(H)}
\ge
e^{-\Phi(H)}
=
W_1(H)$, and therefore also $Z_\beta\ge Z_1\ge e^{-2\tau}$. Combining both cases yields then $Z_\beta\ge e^{-2\bstar\tau}$.
\end{proof}

The determinant-weighted partition function $Z_\beta$ admits a useful
variational interpretation.  A probability distrubtion \(\mu\) on the Boolean hypercube may concentrate on signings for which the spectral barrier $\Phi$ is small,
but doing so generally incurs an entropy cost relative to the uniform
distribution \(U_n\).  The following elementary Gibbs variational identity
shows that \(-\log Z_\beta\) is captures exactly the optimal tradeoff
between these two quantities.
\begin{lemma}[Gibbs variational identity]\label{lem:gibbs}
Let $A_1,\dots,A_n\in\Sym_d(\R)$ and let  $S_x=\sum_{i=1}^n x_iA_i$.
Suppose that $Z_\beta = \E_{x\sim U_n}W_\beta(S_x) >0$, and consider the probability density $\mu_\beta(x)=2^{-n}e^{-\beta\Phi(S_x)}/Z_\beta$ on $\{-1,1\}^n$. Then
\begin{equation}\label{eq:gibbs}
-\log Z_\beta
=
\inf_\mu
\left\{
D(\mu\|U_n)
+
\beta\,\E_\mu\Phi(S_x)
\right\},
\end{equation}
where the infimum is over probability measures \(\mu\) supported on $\{x\in\{-1,1\}^n:\|S_x\|<1\}.$
Moreover, the infimum is attained by \(\mu_\beta\).
\end{lemma}

\begin{proof}
First note that, 
\begin{align*}
D(\mu\|\mu_\beta)
&=
\sum_x
\mu(x)\log\frac{\mu(x)}{\mu_\beta(x)} =
\sum_x
\mu(x)\log\frac{\mu(x)}{2^{-n}}
+
\beta\sum_x\mu(x)\Phi(S_x)
+
\log Z_\beta \\
&=
D(\mu\|U_n)
+
\beta\,\E_\mu\Phi(S_x)
+
\log Z_\beta.
\end{align*}
Rearranging gives $D(\mu\|U_n)+\beta\,\E_\mu\Phi(S_x)=-\log Z_\beta+D(\mu\|\mu_\beta)\ge-\log Z_\beta$. Thus the right-hand side of \eqref{eq:gibbs} is bounded below by
\(-\log Z_\beta^{\mathrm B}\).
Finally, taking \(\mu=\mu_\beta\) gives
\(D(\mu_\beta\|\mu_\beta)=0\), and hence equality.  Therefore the
infimum is attained by \(\mu_\beta\), proving
\eqref{eq:gibbs}.
\end{proof}

\subsection{Curvature Bounds}

In our arguments we will need a second,  auxiliary version of the log-determinant
barrier. For $|x|<1$, define $\ell(x)=-\log(1-x^2)$ and $u(x)=32\ell(x)-28x^2=4x^2+32\sum_{k\ge2}x^{2k}/k$.
We consider the spectral potential
\begin{equation}\label{eq:def-spectral-u}
\mathcal U(H) = 
 \begin{cases}
 \Tr u(H),&\|H\|<1,\\
 +\infty,&\|H\|\ge1,
 \end{cases}
\end{equation}
which  is nonnegative, even, and convex. It is related to the barrier function $\Phi$ via $\mathcal U(H)
=
32\Phi(H)-28\Tr H^2$.

\begin{lemma}\label{lem:gaussian-input}
There are universal $\eta,C_R>0$ such that the following holds.
Let $X=\sum_i g_iA_i$, where $g_i$ are independent standard
Gaussians, $A_i\in\Sym_d(\R)$, and
$\|\sum_iA_i^2\|\le\eta$. Consider the tilted Gaussian measure,
\begin{equation}\label{eq:gaussian-tilt}
 \dd\mu(g)=
 \frac{e^{-\mathcal U(X(g))}}{\E_\gamma e^{-\mathcal U(X)}}\,
 \dd\gamma(g).
\end{equation}
Then, for every  symmetric matrix $H$ and $\rho\in\{-1,1\}$ we have that,
\begin{equation}\label{eq:reciprocal-input}
 \E_\mu\Tr\!\left[(I+\rho X)^{-1}H(I+\rho X)^{-1}H\right]
 \le C_R\|H\|_{\Fro}^2.
\end{equation}
\end{lemma}
\begin{proof}
We deduce the statement from
\citet[Proposition 3.8]{AS26}.
In the notation of that proposition, choose $q= 2$, $\alpha=32$ and $b=4$. For this choice, the scalar potential appearing in \cite{AS26} is $4x^2+32\sum_{k\ge2}x^{2k}/k=32\bigl(-\log(1-x^2)\bigr)-28x^2=u(x)$.
Hence also the corresponding spectral potential is precisely $\Tr u(X)=\mathcal U(X),$
and the normalized tilted Gaussian measure in \cite[Proposition 3.8]{AS26} coincides with the measure \(\mu\) defined in \eqref{eq:gaussian-tilt}. Moreover, the parameter requirements of
\cite[Proposition 3.8]{AS26} are satisfied for this choice: $\alpha\ge b$, $\alpha\ge32$, and $\alpha b\ge128$.
Since \(q=2\), its small-variance assumption
\(q^2\eta\le c_{\mathrm P}\) is ensured by taking $\eta\le \frac{c_{\mathrm P}}{4}$.
Now let \(\rho\in\{-1,1\}\) and  \(H\in\Sym_d(\R)\), and set \(h=\vecop(H)\). Applying
\cite[Proposition 3.8]{AS26} with $\sigma=\tau=\rho$
and with the constant matrix-valued field $F(g)\equiv H$ gives then,
\[
\E_\mu\!\left[
 h^{\mathsf T}
 \Bigl(
   (I+\rho X)^{-1}\otimes (I+\rho X)^{-1}
 \Bigr)
 h
\right]
\le
C_R\|h\|_2^2
\]
for a universal constant \(C_R>0\).
Finally, using  that \(H\) and \(I+\rho X\) are symmetric and  the
standard vectorization identity $\vecop(H)^{\mathsf T}
(A\otimes B)\vecop(H)
=
\Tr\!\left(BHA^{\mathsf T}H^{\mathsf T}\right)$, we obtain $\E_\mu\bigl[h^{\mathsf T}\bigl((I+\rho X)^{-1}\otimes(I+\rho X)^{-1}\bigr)h\bigr]=\E_\mu\Tr\bigl[(I+\rho X)^{-1}H(I+\rho X)^{-1}H\bigr]$, which concludes the proof.
\end{proof}

The following lemma provides an estimate on the average curvature of the potential \(\mathcal U\).
\begin{lemma}[Curvature bound]\label{lem:hessian}
Under the hypotheses of Lemma~\ref{lem:gaussian-input}, for every symmetric $H$ we have, 
\begin{equation}\label{eq:hessian}
 \E_\mu D^2\mathcal U(X)[H,H]\le 64C_R\|H\|_{\Fro}^2.
\end{equation}
\end{lemma}

\begin{proof}
Recall that, $\mathcal U(X)
=
32\Phi(X)-28\Tr(X^2)$ and $\Phi(X)
=
-\log\det(I-X^2)
=
-\log\det(I-X)-\log\det(I+X).$
Since  \(\mu\) is supported on
\(\{X:\|X\|<1\}\), the matrices \(I\pm X\) are invertible
\(\mu\)-almost surely.  For \(\rho\in\{-1,1\}\), write $R_\rho=(I+\rho X)^{-1}$.
Consider $f_\rho(t)
=
-\log\det\bigl(I+\rho(X+tH)\bigr)$.  Using $\frac{\dd}{\dd t}\log\det A(t)
=
\Tr\!\left(A(t)^{-1}A'(t)\right)$,
we obtain $f_\rho'(t)=-\rho\,\Tr\bigl[(I+\rho(X+tH))^{-1}H\bigr]$.
Moreover, using $\frac{\dd}{\dd t}A(t)^{-1}
=
-A(t)^{-1}A'(t)A(t)^{-1},$
gives $f_\rho''(0)
=
\Tr(R_\rho H R_\rho H)$.
Therefore, we obtain $D^2\Phi(X)[H,H]
=
\sum_{\rho=\pm1}
\Tr(R_\rho H R_\rho H).$
Moreover, since  $D^2\Tr(X^2)[H,H] = 2\|H\|_{\Fro}^2$, we have
\[
D^2\mathcal U(X)[H,H]
=
32\sum_{\rho=\pm1}
\Tr(R_\rho H R_\rho H)
-
56\|H\|_{\Fro}^2.
\]
Now Lemma~\ref{lem:gaussian-input} gives $\E_\mu\Tr(R_\rho HR_\rho H)\le C_R\|H\|_{\Fro}^2$, and therefore $\E_\mu D^2\mathcal U(X)[H,H]\le(64C_R-56)\|H\|_{\Fro}^2\le64C_R\|H\|_{\Fro}^2$, which concludes the proof.
\end{proof}

\section{Proof of the Boolean Matrix Small-Ball Theorem \ref{thm:main}}
\subsection{A Functional Geometric Criterion for Full Signings}
\label{sec:signing}

For \(u\in\R^N\), we denote with $\partial_u \varphi = \langle u,\nabla\varphi\rangle$ the directional derivative of a smooth test function \(\varphi\).
If \(p\in L^1(\R^N)\), its distributional derivative in the direction
\(u\) is defined by $\langle D_u p,\varphi\rangle=-\int_{\R^N} p(x)\,\partial_u\varphi(x)\,\dd x$ for $\varphi\in C_c^1(\R^N)$.
We denote with \(BV(\R^N)\)  the space of functions of bounded
variation, namely, functions \(p\in L^1(\R^N)\) whose distributional
gradient \(Dp\) is a finite \(\R^N\)-valued Radon measure.  Equivalently,
each distributional directional derivative \(D_u p\) is a finite signed
Radon measure. Let \(p\in BV(\R^N)\), we denote its total variation by $V_u(p)
:=
|D_u p|(\R^N)$.
Moreover, we write \(W^{1,1}(\R^N)\) for the Sobolev space of integrable functions
whose first weak derivatives are also integrable.
Thus, for \(p\in W^{1,1}(\R^N)\) we have that \(\nabla p\in L^1(\R^N;\R^N)\), and the
distributional derivative \(D_u p\) is absolutely continuous with
respect to Lebesgue measure, with density $\partial_u p=\langle u,\nabla p\rangle$, so that $p$ has total variation $V_u(p)
= \int_{\R^N}
|\partial_u p(x)|\,\dd x$.
In particular, we have that $W^{1,1}(\R^N)\subset BV(\R^N)$. Moreover, for  $K\subseteq \R^N$ we let \(BV_K(\R^N)\) denote the functions in
\(BV(\R^N)\) that vanish Lebesgue-a.e. on \(\R^N\setminus K\).
Similarly, we use \(W^{1,1}_K(\R^N)\) to denote those functions in
\(W^{1,1}(\R^N)\) that vanish Lebesgue-a.e. outside \(K\).

The following result of \citet[Corollary 2.3]{GFL26}  provides a geometric criterion for symmetric convex sets $K$ under which every family of sufficiently short vectors admits a signing whose signed sum lies in  \(K\).
It may be viewed as a functional analogue of the classical convex geometric  vector-balancing principles, in the spirit of \cite{Gluskin1989Extremal,Ban98}.

\begin{theorem}[Vector balancing under bounded directional variation]\label{thm:tv-input}
Let $K\subset\R^N$ be a  bounded open  convex set satisfying $K=-K$. Suppose that for some $\kappa>0$ there exists  a probability density  $p\in BV_K(\R^N)$  satisfying 
 $V_u(p)\le\kappa\|u\|_2$ for
all $u\in\R^N$.  Then every family $a_1,\ldots,a_m\in\R^N$ with
$\kappa\max_i\|a_i\|_2\le1/3$ admits a coloring $x\in\{-1,1\}^m$ such that
$\sum_i x_ia_i\in K$.
\end{theorem}

\subsubsection{A Fisher Information Criterion}

Recall that for a probability density \(p\in W^{1,1}(\R^N)\) its Fisher
information matrix, whenever finite, is defined by
\[
J(p)
=
\int_{\{p>0\}}
\frac{\nabla p(x)\nabla p(x)^{\mathsf T}}{p(x)}\,\dd x.
\]
The following result below establishes a useful integrality criterion for symmetric convex sets based on Fisher information matrices. 
\begin{proposition}[Fisher information criterion]\label{prop:fisher-vertex}
Let $K\subset\R^N$ be an open convex set satisfying $K=-K$.
Suppose there exists a probability density function
\(p\in W^{1,1}_K(\R^N)\) and a symmetric positive-definite matrix \(M\) such that 
$J(p)\preceq M$.
   Then, every family  $a_1,\ldots,a_n\in\R^N$  with  $\max_i a_i^{\mathsf T}Ma_i<1/9$  admits a coloring $x\in\{-1,1\}^n$ such that $\sum_i x_i a_i\in K$.
\end{proposition}
\begin{proof}
Let $A=M^{1/2}$. We apply the linear change of variables $y=Ax$  and denote by \(p^A\) the pushforward of \(p\). Thus $p^A(y) =\frac{1}{\det A}\,p(A^{-1}y)$.
Note that, since \(p\) is a probability density, so is \(p^A\), and since
\(p\) vanishes a.e. outside \(K\), the density \(p^A\) vanishes
a.e. outside \(AK\).

Next, we compute how the Fisher information transforms under this change of variables. Using that
$\nabla p^A(y)
=
\frac{1}{\det A}\,
A^{-1}\nabla p(A^{-1}y)$
 we  obtain,
\begin{align*}
J(p^A)
&=
\int_{\{p^A>0\}}
\frac{\nabla p^A(y)\nabla p^A(y)^{\mathsf T}}
     {p^A(y)}\,\dd y \\
&=
A^{-1}
\biggl[
\int_{\{p>0\}}
\frac{\nabla p(x)\nabla p(x)^{\mathsf T}}
     {p(x)}\,\dd x
\biggr]
A^{-1} \\
&=
A^{-1}J(p)A^{-1}.
\end{align*}
Moreover, since \(J(p)\preceq M=A^2\), it follows that $J(p^A)\preceq A^{-1}MA^{-1}=I$.
This Fisher-information bound translates to a bound of the total variation of $p^A$ for all directions $u\in\R^N$.
Indeed, for  \(u\in\R^N\), Cauchy--Schwarz gives
\begin{align}
V_u(p^A)
&=
\int_{\R^N}
\bigl|\langle u,\nabla p^A(y)\rangle\bigr|\,\dd y \notag\\
&=
\int_{\{p^A>0\}}
\frac{\bigl|\langle u,\nabla p^A(y)\rangle\bigr|}
     {\sqrt{p^A(y)}}
\sqrt{p^A(y)}\,\dd y \notag\\
&\le
\biggl(
\int_{\{p^A>0\}}
\frac{\langle u,\nabla p^A(y)\rangle^2}
     {p^A(y)}\,\dd y
\biggr)^{1/2}
\biggl(
\int_{\R^N}p^A(y)\,\dd y
\biggr)^{1/2} \notag\\
&=
\bigl(u^{\mathsf T}J(p^A)u\bigr)^{1/2}
\le
\|u\|_2.
\label{eq:variation-whitened}
\end{align}

We next identify the vectors to which
Theorem~\ref{thm:tv-input} will be applied. To this end, for each  
\(a_i\in\R^N\), define $b_i:=Aa_i$ for $i=1,\ldots,n.$
Since \(A^2=M\), we have $\|b_i\|_2^2
=
a_i^{\mathsf T}A^2a_i
=
a_i^{\mathsf T}Ma_i.$ Therefore, since $\max_{1\le i\le n}a_i^{\mathsf T}Ma_i<\frac19$, we obtain that $\max_{1\le i\le n}\|b_i\|_2<\frac13.$
Moreover, note that the set \(AK\) is open symmetric convex set,   but it need not be
bounded.
We therefore use  the following truncation argument. 
Choose a smooth function \(\eta:[0,\infty)\to[0,1]\) such that $\eta(t)=1$ for $0\le t\le1$ and
$\eta(t)=0$ for $t\ge2$,
and set $\chi_R(x) = \eta\bigl(\|x\|_2/R\bigr)$.
Then \(\chi_R\in C_c^\infty(\R^N)\) is a smooth radial function with $\chi_R=1$ on $B(0,R)$, $\chi_R=0$ on $\R^N\setminus B(0,2R)$, and $\|\nabla\chi_R\|_\infty\le c/R$, where $c:= \|\eta'\|_\infty$.
Moreover, define $m_R=\int_{\R^N}\chi_R(y)p^A(y)\,\dd y$ and $p_R=\chi_Rp^A/m_R$, and let $K_R=AK\cap B(0,2R)$.
Since \(0\le\chi_R\le1\) and \(\chi_R(y)\to1\) pointwise as
\(R\to\infty\), dominated convergence shows $m_R\longrightarrow1$.
In particular, \(m_R>0\) for all sufficiently large \(R\).
For all such \(R\), \(K_R\) is a nonempty bounded open symmetric convex set. Moreover, \(p_R\) is a probability density that
vanishes a.e. outside \(K_R\). Now, since \(p^A\in W^{1,1}(\R^N)\) and
\(\chi_R\) is smooth, the product rule gives $\partial_u(\chi_R p^A)=\chi_R\,\partial_u p^A+p^A\,\partial_u\chi_R$, and therefore
\begin{align*}
V_u(p_R)
&\le
\frac{1}{m_R}
\biggl(
\int_{\R^N}
\chi_R\bigl|\partial_u p^A\bigr|\,\dd y
+
\int_{\R^N}
p^A\bigl|\partial_u\chi_R\bigr|\,\dd y
\biggr) \\
&\le
\frac{1}{m_R}
\biggl(
V_u(p^A)
+
\frac{c}{R}\|u\|_2
\int_{\R^N}p^A(y)\,\dd y
\biggr).
\end{align*}
Moreover, using \eqref{eq:variation-whitened} we have that $V_u(p_R)\le\kappa_R\|u\|_2$ with $\kappa_R=(1+c/R)/m_R$, and $\kappa_R\longrightarrow1$ as \(R\to\infty\). 
We may therefore choose  \(R\) sufficiently large such that $\kappa_R\max_i\|b_i\|_2<\frac13.$

Finally, we may now apply Theorem~\ref{thm:tv-input} and obtain that there are signs  
\(x_1,\ldots,x_n\in\{-1,1\}\) such that $\sum_{i=1}^n x_i b_i=A\sum_{i=1}^n x_i a_i\in K_R\subset AK$. But, since \(A\) is invertible, it follows that \(\sum_{i=1}^n x_i a_i\in K\),
which concludes the proof.
\end{proof}

\subsubsection{Constructing a Density with Controlled Fisher Information}
Let \(A_1,\dots,A_n\in\Sym_d(\R)\), and write $S_x=\sum_{i=1}^n x_iA_i$ and $\Omega=\{x\in\R^n:\|S_x\|<1\}$.
The goal of this section is to construct a probability density function on $ \Omega$  whose Fisher information matrix can be controlled well enough to apply the criterion of  Proposition~\ref{prop:fisher-vertex} and   hence show the existence of a coloring $x\in\{-1,1\}^n$ with $\|S_x\|<1$. To this end, for $t>0$ we consider 
\begin{equation}\label{eq:coefficient-density}
 p_t(x)=\frac{(2\pi t^2)^{-n/2}}{\mathcal Z_t}
       e^{-\|x\|_2^2/(2t^2)-\mathcal U(S_x)},
\end{equation}
where $\mathcal{U}$ is defined as in \eqref{eq:def-spectral-u}. The following lemma below establishes precisely the aforementioned  control on the Fisher information matrix $J(p_t)$ of $p_t$ in terms of the Gram matrix $[G_{ij}=\Tr(A_iA_j)]_{i,j=1}^n$.
\begin{lemma}[Fisher information bound of the barrier density]\label{lem:fisher}
Let $p_t$ be defined as in \eqref{eq:coefficient-density} and let $G=[G_{ij}=\Tr(A_iA_j)]_{i,j=1}^n$ be the Gram matrix associated with the coefficient matrices $A_1, \dots, A_n\in\Sym_d(\R)$.
Then,   \(p_t\in W^{1,1}(\R^n)\), and \(p_t\) has finite  Fisher
information matrix. Moreover, for every \(u\in\R^n\) we have,
\begin{equation}\label{eq:fisher-exact}
u^{\mathsf T}J(p_t)u
=
t^{-2}\|u\|_2^2
+
\E_{p_t}
D^2\mathcal U(S_x)[S_u,S_u].
\end{equation}
If in addition $\|\sum_{i=1}^n A_i^2\|\le t^{-2}\eta$, 
then
\begin{equation}\label{eq:fisher-bound}
J(p_t)
\preceq
t^{-2}I+64C_RG,
\end{equation}
where \(\eta\) and \(C_R\) are the universal constants from
Lemma~\ref{lem:gaussian-input}.
\end{lemma}
\begin{proof}
We divide the proof into three steps.

\paragraph{Step 1: Regularity of the density.}

Define $\Psi(H):=e^{-\mathcal U(H)}$ for $H\in\Sym_d(\R)$,
with the convention \(e^{-\infty}=0\). Thus \(\Psi(H)=0\) whenever
\(\|H\|\ge1\), and $\Psi(H)
=
\det(I-H^2)^{32}e^{28\Tr H^2}$ on the open unit operator norm ball.

We first verify that \(\Psi\) is \(C^2\) across the
boundary \(\{H:\|H\|=1\}\). To this end, let $g(H)=\det(I-H^2)$.
Then \(g\) is a smooth function of the matrix entries with $D(g^{32})[H]=32g^{31}Dg[H]$ and $D^2(g^{32})[H,H]=32\cdot31\,g^{30}Dg[H]^2+32g^{31}D^2g[H,H]$.
Since \(g\), its first two derivatives, and
\(e^{28\Tr H^2}\) together with its first two derivatives are bounded
on the compact set $\{H\in\Sym_d(\R):\|H\|\le1\},$
it follows that \(\Psi\), \(D\Psi\), and \(D^2\Psi\) all tend to zero
as \(H\) approaches the boundary \(\{H:\|H\|=1\}\) from within the unit   operator norm ball, which shows that \(\Psi\in C^2(\Sym_d(\R))\).

Moreover, the function \(V:x\mapsto\Psi(S_x)\) is  \(C^2\)
on \(\R^n\), since the map $x\longmapsto S_x=\sum_i x_iA_i$
is linear. 
Now, since the function $V$ and its first two derivatives are bounded on
\(\R^n\), $p_t(x)
=
\frac{(2\pi t^2)^{-n/2}}{\mathcal Z_t}
e^{-\|x\|_2^2/(2t^2)}\Psi(S_x)$
belongs to \(C^2(\R^n)\) and all of its derivatives of order at most two are dominated by an
integrable function of the form $C(1+\|x\|_2^2)e^{-\|x\|_2^2/(2t^2)}.$ In particular,
$p_t\in W^{2,1}(\R^n)\subset W^{1,1}(\R^n)$.

Finally, we verify next that \(p_t\) has finite Fisher information.  To this end, note that $\partial_u\log p_t(x)=-t^{-2}\langle x,u\rangle-D\mathcal U(S_x)[S_u]$.
Let $m(H)=1-\|H\|.$ Since $D\mathcal U(H)[Y]=32D\Phi(H)[Y]-56\Tr (HY)$, we have that $|D\mathcal U(H)[Y]|\le 64\sqrt d\,m(H)^{-1}\|Y\|_{\Fro}+56\sqrt d\,\|Y\|_{\Fro}\le C_d\,m(H)^{-1}\|Y\|_{\Fro}$,
where \(C_d=64 \sqrt{d}\).
Moreover, using $\det(I-H^2)\le1-\|H\|^2\le2m(H)$ we have $e^{-\mathcal U(H)}=\det(I-H^2)^{32}e^{28\Tr H^2}\le C_d\,m(H)^{32}$.
Therefore, for every  \(u\) we obtain that, 
\begin{equation}\label{eq:lem-fisher-integrable}
    p_t(x)\bigl(\partial_u\log p_t(x)\bigr)^2
\le
C_{A,d,t,u}\,
e^{-\|x\|_2^2/(2t^2)}
\bigl(1+\|x\|_2^2+m(S_x)^{30}\bigr)
\mathbf 1_{\Omega}(x).
\end{equation}
The right-hand side of \eqref{eq:lem-fisher-integrable} is integrable and hence $u^{\mathsf T}J(p_t)u=\int_{\{p_t>0\}}p_t(x)\bigl(\partial_u\log p_t(x)\bigr)^2\,\dd x<\infty$
for every \(u\), which shows that \(J(p_t)\) is finite.

\paragraph{Step 2: The Fisher information identity.}
By Step~1, \(p_t\in C^2(\R^n)\) and $p_t,\ \partial_u p_t,\ \partial_u^2p_t\in L^1(\R^n).$
We next show that the integral of the second directional derivative
of \(p_t\) over \(\R^n\) vanishes. To this end, choose
\(\chi\in C_c^\infty(\R^n)\) with $0\le \chi\le1$, $\chi=1$ on $B(0,1)$, and $\chi=0$ outside $B(0,2)$,
and set $\chi_R(x)=\chi(x/R).$
Then $\|\nabla\chi_R\|_\infty\le\|\nabla\chi\|_\infty/R.$
Since \(\chi_R\partial_u p_t\) has compact support, integration by
parts gives $\int_{\R^n}\chi_R\,\partial_u^2p_t\,\dd x=-\int_{\R^n}\partial_u\chi_R\,\partial_u p_t\,\dd x$.
Hence
\begin{equation}\label{eq:int-parts}
    \biggl|
\int_{\R^n}
\chi_R\,\partial_u^2p_t\,\dd x
\biggr|
\le
\frac{\|\nabla\chi\|_\infty\|u\|_2}{R}
\int_{\R^n}|\partial_u p_t(x)|\,\dd x.
\end{equation}
Because \(\partial_u p_t\in L^1(\R^n)\), the right-hand side of \eqref{eq:int-parts} tends
to zero as \(R\to\infty\).  On the other hand,
\(\partial_u^2p_t\in L^1(\R^n)\) and
\(\chi_R\to1\) pointwise, so dominated convergence yields
\begin{equation}\label{eq:lem-fish-second}
    \int_{\R^n}\partial_u^2p_t(x)\,\dd x=0.
\end{equation}

Now, using $\partial_u p_t=p_t\,\partial_u\log p_t,$
we have that $\partial_u^2p_t=p_t\bigl(\partial_u\log p_t\bigr)^2+p_t\,\partial_u^2\log p_t$.
Integrating over \(\Omega\) and using \eqref{eq:lem-fish-second} we obtain
\begin{equation}\label{eq:lem-fish-fourth}
    \int_{\Omega}
p_t(x)\bigl(\partial_u\log p_t(x)\bigr)^2\,\dd x
=
-\int_{\Omega}
p_t(x)\partial_u^2\log p_t(x)\,\dd x.
\end{equation}
Moreover,
\begin{equation}\label{eq:lem-fish-third}
    \int_{\Omega}
p_t(x)\bigl(\partial_u\log p_t(x)\bigr)^2\,\dd x
=
\int_{\Omega}
\frac{\langle u,\nabla p_t(x)\rangle^2}{p_t(x)}\,\dd x
=
u^{\mathsf T}J(p_t)u.
\end{equation}
Now, recall that for \(x\in\Omega\), $\log p_t(x)=\mathrm{const}-\|x\|_2^2/(2t^2)-\mathcal U(S_x)$.
Using that $\partial_u^2\,\|x\|_2^2/(2t^2)=t^{-2}\|u\|_2^2$ and $\partial_u^2\mathcal U(S_x)=D^2\mathcal U(S_x)[S_u,S_u]$,
we obtain
\begin{equation}\label{eq:log-hessian-pt}
-\partial_u^2\log p_t(x)
=
t^{-2}\|u\|_2^2
+
D^2\mathcal U(S_x)[S_u,S_u].
\end{equation}

Moreover, using \eqref{eq:lem-fish-fourth} and \eqref{eq:lem-fish-third}, we obtain from
 \eqref{eq:log-hessian-pt},
\begin{align*}
u^{\mathsf T}J(p_t)u
&=
\int_{\Omega}
p_t(x)
\Bigl[
t^{-2}\|u\|_2^2
+
D^2\mathcal U(S_x)[S_u,S_u]
\Bigr]
\,\dd x \\
&=
t^{-2}\|u\|_2^2
\int_{\Omega}p_t(x)\,\dd x
+
\int_{\Omega}
p_t(x)
D^2\mathcal U(S_x)[S_u,S_u]
\,\dd x\\
&= 
t^{-2}\|u\|_2^2
+
\E_{p_t}
D^2\mathcal U(S_x)[S_u,S_u],
\end{align*}
which proves \eqref{eq:fisher-exact}.

\paragraph{Step 3: Bounding the Fisher information matrix.}
Assume now that $\Bigl\|\sum_{i=1}^n A_i^2\Bigr\|\le t^{-2}\eta.$ Consider the change of variables $x=tg$, $g\in\R^n$, so that \(\dd x=t^n\,\dd g\) and $(2\pi t^2)^{-n/2}e^{-\|x\|_2^2/(2t^2)}\,\dd x=(2\pi)^{-n/2}e^{-\|g\|_2^2/2}\,\dd g=\dd\gamma(g)$,
where \(\gamma\) denotes the standard Gaussian measure on \(\R^n\).

Moreover, we let $M_i:=tA_i$ so that $S_{tg}
=
\sum_{i=1}^n tg_iA_i
=
\sum_{i=1}^n g_iM_i$. Now, note that the normalizing constant of $p_t$ can be expressed as $\mathcal Z_t=\E_{g\sim\gamma}\exp\bigl\{-\mathcal U\bigl(\sum_{i=1}^n g_iM_i\bigr)\bigr\}$,
so that under the change of variables \(x=tg\), the probability
measure with density \(p_t\) is given by
\begin{equation}\label{eq:lem-fish-law}
    \dd\mu(g)
=
\frac{
\exp\bigl\{-\mathcal U(X(g))\bigr\}
}{
\E_\gamma\exp\bigl\{-\mathcal U(X)\bigr\}
}
\,\dd\gamma(g),
\qquad
X(g):=\sum_{i=1}^n g_iM_i.
\end{equation}
Therefore, the law of \(S_x\) under \(p_t\) is exactly the tilted Gaussian
law from \eqref{eq:lem-fish-law} with coefficient matrices $M_i=tA_i$.
Finally, since $\|\sum_{i=1}^nM_i^2\|=t^2\|\sum_{i=1}^nA_i^2\|\le\eta$, we may apply 
Lemma~\ref{lem:hessian} to obtain $\E_{p_t}D^2\mathcal U(S_x)[S_u,S_u]\le 64 C_R\|S_u\|_{\Fro}^2$. Moreover, by the definition of the Gram matrix \(G\), $\|S_u\|_{\Fro}^2=\Tr\bigl(\sum_i u_iA_i\bigr)^2=\sum_{i,j}u_iu_j\Tr(A_iA_j)=u^{\mathsf T}Gu$.
Combining this with \eqref{eq:fisher-exact}, we obtain
\[
u^{\mathsf T}J(p_t)u
\le
t^{-2}\|u\|_2^2
+
64 C_R\,u^{\mathsf T}Gu
=
u^{\mathsf T}\bigl(t^{-2}I+64C_RG\bigr)u.
\]
Since this holds for every \(u\in\R^n\), $J(p_t)\preceq t^{-2}I+64C_RG$,
which is \eqref{eq:fisher-bound}.
\end{proof}

\subsubsection{ Warm-Up: A Variance--Frobenius Coloring Criterion}
We first illustrate how the preceding Fisher-information framework 
produces a full coloring in a simple setting.  Here we work with the
usual matrix variance $v=\Bigl\|\sum_{i=1}^n A_i^2\Bigr\|$, 
and the maximal coefficient energy $q=\max_i\Tr(A_i^2).$
More specifically, a bound on \(v\)
ensures that the Fisher-information estimate $J(p_t)\preceq t^{-2}I+LG$
is available, and since $\bigl(t^{-2}I+LG\bigr)_{ii}
=
t^{-2}+L\Tr(A_i^2),$
a control on  $q=\max_i\Tr(A_i^2)$ guarantees that these diagonal entries are below the threshold \(1/9\) as required by the Fisher-information signing criterion of  \Cref{prop:fisher-vertex}.  Later we apply the same mechanism only to the off-diagonal remainders of the matrices, which leads to the stronger parameter \(\nu\) appearing in the main theorem.
\begin{theorem}[Variance--Frobenius coloring criterion]
\label{thm:existence}
There exist universal constants \(v_E,q_E>0\) such that every
symmetric family \(A_1,\ldots,A_n\in\Sym_d(\R)\) satisfying
\[
v:=\Bigl\|\sum_{i=1}^n A_i^2\Bigr\|\le v_E,
\qquad
q:=\max_{1\le i\le n}\Tr(A_i^2)\le q_E
\]
admits a full coloring \(x\in\{-1,1\}^n\) for which $\Bigl\|\sum_{i=1}^n x_iA_i\Bigr\|<1.$
Consequently, there exists a universal constant \(C_E>0\) such that
\begin{equation}\label{eq:thm-exist-disc}
    \disc(A)
:=
\min_{x\in\{-1,1\}^n}
\Bigl\|\sum_{i=1}^n x_iA_i\Bigr\|
\le
C_E\sqrt{v+q}
\end{equation}
for every finite symmetric matrix family.
\end{theorem}
\begin{proof}
Let \(\eta\) and \(C_R\) be the universal constants from
\cref{lem:gaussian-input}.  Set $L:=64C_R$, $v_E:=\frac{\eta}{64}$, and $q_E:=\frac{1}{64L}.$ We now apply \cref{lem:fisher} to the density \(p_t\) defined in
\eqref{eq:coefficient-density}, with the choice \(t=8\).
Indeed, since
\[
v=\Bigl\|\sum_{i=1}^n A_i^2\Bigr\|
\le
v_E
=
\frac{\eta}{64}
=
\frac{\eta}{8^2},
\]
we obtain using \cref{lem:fisher} that $J(p_8)
\preceq
\frac1{64}I+64C_R\,G$,  where $G_{ij}=\Tr(A_iA_j).$
Let $M:=\frac1{64}I+LG$.  Since  \(G\succeq0\), we have
\(M\succ0\).  Moreover, for every \(i\),
\[
M_{ii}
=
\frac1{64}+L\Tr(A_i^2)
\le
\frac1{64}+Lq_E
=
\frac1{32}
<
\frac19.
\]
Now, let  $\Omega
=
\biggl\{
x\in\R^n:
\Bigl\|\sum_{i=1}^n x_iA_i\Bigr\|<1
\biggr\}$. Then $\Omega$ is a open symmetric convex set. Moreover, 
by construction, \(p_8\) is a probability density in
\(W^{1,1}_\Omega(\R^n)\), and we have shown that
\[
J(p_8)\preceq M,
\qquad
M\succ0,
\qquad
\max_i M_{ii}<\frac19,
\]
and therefore, \cref{prop:fisher-vertex} yields that $x\in\Omega\cap\{-1,1\}^n$, or equivalently, $\Bigl\|\sum_{i=1}^n s_iA_i\Bigr\|<1$.

We now derive the discrepancy bound \eqref{eq:thm-exist-disc} for an arbitrary nonzero family of matrices by rescaling the coefficient matrices so that the hypotheses of the first part are satisfied. More specifically, let $r
=
\max\biggl\{
\sqrt{\frac{v}{v_E}},
\sqrt{\frac{q}{q_E}}
\biggr\}$ and set $\widetilde A_i=\frac{A_i}{r}$.
For the rescaled family we have the following bounds for associated  parameters $\widetilde v,\widetilde q$, 
\[
\Bigl\|\sum_i\widetilde A_i^2\Bigr\|
=
\frac{v}{r^2}
\le v_E, \qquad
\max_i\Tr(\widetilde A_i^2)
=
\frac{q}{r^2}
\le q_E.
\]
The first part of the theorem therefore gives signs \(x_i\) such that
$\Bigl\|\sum_i x_i\widetilde A_i\Bigr\|<1$, so that by 
multiplying by \(r\), $\disc(A)\le\Bigl\|\sum_i s_iA_i\Bigr\|<r$.
Finally, since $r
\le
\max\bigl\{v_E^{-1/2},q_E^{-1/2}\bigr\}\sqrt{v+q}$ we choose $C_E
=
\max\bigl\{v_E^{-1/2},q_E^{-1/2}\bigr\}$, which concludes the proof.
\end{proof}
\subsection{From Existence of Coloring to a Small-Ball Bound}
\label{sec:amplification}

The previous section, \cref{sec:signing}, provides a framework for
obtaining a single full coloring. We now show that a uniform existence
statement of this kind can be amplified into a quantitative Boolean
small-ball bound. The amplification argument is purely combinatorial
and is independent of the analytic or geometric method used to obtain
the initial coloring. In fact, the argument upgrades a universal full-signing
statement to a determinant-weighted Boolean small-ball estimate.

The following definition is ad-hoc and will be convenient to use in the replica amplification argument below. 
\begin{definition}\label{def:E}
For \(a,b>0\), we say that the \emph{\((a,b)\)-coloring property}
holds, and write \(\mathsf E(a,b)\), if every finite family of
symmetric matrices \(C_1,\ldots,C_N\), in arbitrary matrix dimension,
satisfying $\Bigl\|\sum_i C_i^2\Bigr\|\le a$ and $\max_i \Tr(C_i^2)\le b$
admits a coloring \(x\in\{-1,1\}^N\) such that $\Bigl\|\sum_i x_iC_i\Bigr\|\le1.$
\end{definition}

\begin{theorem}[Replica amplification: from existence-to-small-ball]
\label{thm:amplification}
Assume that \(\mathsf E(a,b)\) holds for some \(a,b>0\), and set
\[
A=4+4a,
\qquad
h=\frac{4A\log 2}{b},
\qquad
C_{a,b}
=
\max\biggl\{2,h+\frac{4}{3a}\biggr\}.
\]
Let \(A_1,\ldots,A_n\) be a symmetric matrix family with
\[
v:=\Bigl\|\sum_iA_i^2\Bigr\|\le\frac a4,
\qquad
q:=\max_i\Tr(A_i^2)\le\frac{b}{2A},
\]
and let $S_x:=\sum_{i=1}^n x_iA_i$ and $\tau=\sum_i\Tr(A_i^2)$.
Then, for every \(\beta>0\),
\begin{equation}\label{eq:amplification-partition}
\E_{x\sim U_n} W_\beta(S_x) 
\ge
\exp\bigl\{-C_{a,b}\max(1,\beta)\tau\bigr\},
\end{equation}
where $W_\beta$ is defined as in \eqref{eq:barrier}.
Moreover, if additionally  \(\tau\ge\frac12\), then 
\begin{equation}\label{eq:good-count}
\Prob_{x\sim U_n} \biggl\{
x\in\{-1,1\}^n:
\|S_x\|\le\frac12,\ 
\Tr(S_x^2)\le\frac{\tau}{a}
\biggr\} \ge e^{-h\tau},
\end{equation}
and the conditional law $\mu=U_n(\,\cdot\,|\bigl\{
\|S_x\|\le\frac12,\ 
\Tr(S_x^2)\le\frac{\tau}{a}
\bigr\})$
satisfies
\[
D(\mu\|U_n)\le h\tau,
\qquad
\E_\mu\Phi(S_x)\le\frac{4\tau}{3a},
\]
where $\Phi$ is defined as in \eqref{eq:barrier}.
\end{theorem}

For the proof of Theorem \ref{thm:amplification} we will need the following simple lemma. Note, that the replica construction combines several independently randomized copies of the original family of matrices into a single enlarged instance.  To apply the full-coloring principle to this enlarged family, we must
prevent any one of the new coefficient matrices from accumulating too much
Frobenius energy.  

\begin{lemma}[Balancing weights via permutations]\label{lem:loads}
Let \(w_1,\ldots,w_n\ge0\), and set $\tau=\sum_{i=1}^n w_i$ and $q=\max_{1\le i\le n}w_i$. Then, for every positive integer \(k\), there exist permutations
\(\pi_1,\ldots,\pi_k\) of \([n]\) such that
\begin{equation}\label{eq:loads}
\max_{1\le i\le n}
\sum_{j=1}^k w_{\pi_j(i)}
\le
\frac{k\tau}{n}+q.
\end{equation}
\end{lemma}
\begin{proof}
We construct the permutations successively. Suppose that
\(\pi_1,\ldots,\pi_r\) have already been chosen, and define the
accumulated weights $L_i^{(r)}
=
\sum_{j=1}^r w_{\pi_j(i)}$ for $1\le i\le n.$
To determine the next permutation, choose orderings
\(i_1,\ldots,i_n\) and \(j_1,\ldots,j_n\) of \([n]\) such that
$L_{i_1}^{(r)}
\le
L_{i_2}^{(r)}
\le
\cdots
\le
L_{i_n}^{(r)}$
and $w_{j_1}
\ge
w_{j_2}
\ge
\cdots
\ge
w_{j_n}$, and  then define $\pi_{r+1}(i_m):=j_m$ for $m=1,\ldots,n$.
Thus the largest remaining weights are assigned to the indices with
the smallest accumulated weights.

We claim that after every round,
\begin{equation} \label{eq:lem-bal-two}
    \max_i L_i^{(r)}-\min_i L_i^{(r)}\le q.
\end{equation}
The claim is trivially holds for \(r=0\). Suppose it holds after \(r\)
rounds. Consider two indices \(i,\ell\) with $L_i^{(r)}\le L_\ell^{(r)}.$
By construction, the newly assigned weights \(a_i,a_\ell\) satisfy $a_i\ge a_\ell.$
Hence
\begin{equation}\label{eq:lem-bal-one}
    \bigl(L_\ell^{(r)}+a_\ell\bigr)
-
\bigl(L_i^{(r)}+a_i\bigr)
=
\bigl(L_\ell^{(r)}-L_i^{(r)}\bigr)
-
(a_i-a_\ell).
\end{equation}
Both terms on the right-hand side of \eqref{eq:lem-bal-one} belong to \([0,q]\), and therefore the
absolute value of their difference is at most \(q\). This proves the
\eqref{eq:lem-bal-two} by induction.

After \(k\) rounds, set $M_k:=\max_{1\le i\le n}L_i^{(k)}$ and $m_k:=\min_{1\le i\le n}L_i^{(k)}$.
By the preceding argument, $M_k-m_k\le q$. 
On the other hand, the average of the accumulated weights is 
\[
\frac1n\sum_{i=1}^n L_i^{(k)}
=
\frac1n\sum_{i=1}^n\sum_{j=1}^k w_{\pi_j(i)}
=
\frac{k}{n}\sum_{i=1}^n w_i
=
\frac{k\tau}{n}.
\]
In particular, $m_k
\le
\frac{k\tau}{n}$, and therefore $\max_{1\le i\le n}
\sum_{j=1}^k w_{\pi_j(i)} = M_k
\le
m_k+q
\le
\frac{k\tau}{n}+q$, 
which is precisely \eqref{eq:loads}.
\end{proof}

\begin{proof}[Proof of \Cref{thm:amplification}]
First note that, if \(\tau=0\), then every \(A_i=0\), and the conclusion is immediate. Moreover, if \(0<\tau\le1/2\), the partition-function estimate follows directly from \cref{lem:small-trace}.  We may therefore assume throughout the remainder of the proof that $\tau\ge\frac12.$

\paragraph{Step 1: Choice of the number of replicas.}
Let $k
=
\left\lfloor
\frac{bn}{2A\tau}
\right\rfloor.$ 
Since $\frac{\tau}{n}
\le
q
\le
\frac{b}{2A},$
we have $\frac{bn}{2A\tau}\ge1$, so that  \(k\ge1\). Moreover, using that  \(\lfloor x\rfloor\ge x/2\) for \(x\ge1\), we obtain that  $\frac{n}{k}
\le
\frac{4A\tau}{b}.$ 

Now, let $w_i:=\Tr(A_i^2)$ for $i=1,\ldots,n$. Since $\sum_{i=1}^n w_i=\tau$
and $\max_i w_i=q$, by    \cref{lem:loads} there are permutations
\(\pi_1,\ldots,\pi_k\) of \([n]\) such that
\begin{equation}\label{eq:replica-balanced}
\max_{1\le i\le n}
\sum_{j=1}^k
\Tr\bigl(A_{\pi_j(i)}^2\bigr)
\le
\frac{k\tau}{n}+q.
\end{equation}

\paragraph{Step 2: Encoding the Frobenius norm constraint.}
For each \(i\), we let $b_i=\vecop(A_i)\in\R^{d^2},$ so that $\|b_i\|_2^2
=
\|A_i\|_{\Fro}^2
=
w_i.$
Moreover, for \(t\in\{-1,1\}\), we define $F_i(t)
=
\begin{pmatrix}
0 & t b_i^{\mathsf T}\\
t b_i & 0
\end{pmatrix}.$
Then, we have $F_i(t)^2
=
\begin{pmatrix}
w_i & 0\\
0 & b_ib_i^{\mathsf T}
\end{pmatrix},$
and in particular $\Tr F_i(t)^2=2w_i$.
Moreover, the lower block of $\sum_iF_i(t_i)^2
=
\begin{pmatrix}
\tau & 0\\
0 & \sum_i b_ib_i^{\mathsf T}
\end{pmatrix}$ is  positive semidefinite and has trace \(\tau\), and
hence has operator norm at most \(\tau\), so that $\|\sum_iF_i(t_i)^2\|=\tau.$
Let $z=\sum_i s_it_i b_i$. Finally, since,  $\sum_i s_iF_i(t_i) =
\begin{pmatrix}
0 & z^{\mathsf T}\\
z & 0
\end{pmatrix}$  has nonzero eigenvalues
\(\pm\|z\|_2\), we  therefore obtain
\begin{equation}\label{eq:arrowhead-recovery}
\left\|\sum_i s_iF_i(t_i)\right\|
=
\left\|\sum_i s_it_iA_i\right\|_{\Fro}.
\end{equation}

\paragraph{Step 3: The enlarged instance.}
Let $t_{j,i}\in\{-1,1\}$ be independent and uniformly distributed random signs for $1\le j\le k$ and $1\le i\le n$. For $i\in[n]$, consider the following familiy of matrices, 
\begin{equation}\label{eq:enlarged}
C_i
=
\bigoplus_{j=1}^k
\biggl(
2t_{j,i}A_{\pi_j(i)}
\ \oplus\
\sqrt{\frac{a}{\tau}}\,
F_{\pi_j(i)}(t_{j,i})
\biggr).
\end{equation}
We verify in the following that the family \(C_1,\ldots,C_n\) satisfies the hypotheses of \(\mathsf E(a,b)\). To this end, note that for each replica \(j\),
$\sum_i
\bigl(2t_{j,i}A_{\pi_j(i)}\bigr)^2
=
4\sum_iA_i^2.$
  Using that $\|\sum_iF_i(t_i)^2\|=\tau$ we have $\|
\frac{a}{\tau}
\sum_i
F_{\pi_j(i)}(t_{j,i})^2\| =a.$
Moreover, since \(C_i\) is block diagonal we have $\left\|\sum_iC_i^2\right\|
=
\max\{4v,a\}
\le a,$
where we used \(v\le a/4\).
Now, note that,
\[
\Tr C_i^2 = \sum_{j=1}^k
\biggl(
4w_{\pi_j(i)}
+
\frac{2a}{\tau}w_{\pi_j(i)}
\biggr) = \biggl(
4+\frac{2a}{\tau}
\biggr)
\sum_{j=1}^k w_{\pi_j(i)}\le \biggl(
4+\frac{2a}{\tau}
\biggr)\biggl(
\frac{k\tau}{n}+q
\biggr),
\]
where we used  \eqref{eq:replica-balanced} in the last step.
Since \(\tau\ge1/2\), we have that $4+\frac{2a}{\tau}
\le
4+4a
=
A$, and therefore $\Tr C_i^2
\le
A\bigl(
\frac{k\tau}{n}+q
\bigr).$ Finally, using that $A\frac{k\tau}{n}\le\frac b2$ and $Aq\le\frac b2$, shows that $\Tr C_i^2\le b$ for every $i$. 

Therefore, the family of matrices $C_1,\dots, C_n$ belong to the class of matrix families for which the property  \(\mathsf E(a,b)\) applies, that is to say,  for every realization of the random masks \(T=(t_{j,i})\), there exists a
coloring $s=s(T)\in\{-1,1\}^n$ such that $\left\|\sum_i s_iC_i\right\|\le1.$ 
However, since  this matrix is in fact block diagonal, every block has operator norm  at most one.  In particular, for each \(j=1,\ldots,k\) we have that 
$\sqrt{\frac{a}{\tau}}
\left\|
\sum_{i=1}^n
s_iF_{\pi_j(i)}(t_{j,i})
\right\|
\le1$ and
$\|
2\sum_i s_it_{j,i}A_{\pi_j(i)}\|
\le1$, and therefore, we have  
\begin{equation}\label{eq:replica-good}
\biggl\|
\sum_i s_it_{j,i}A_{\pi_j(i)}
\biggr\|
\le\frac12,
\qquad
\biggl\|
\sum_i s_it_{j,i}A_{\pi_j(i)}
\biggr\|_{\Fro}^2
\le\frac{\tau}{a},
\end{equation}
where we have used \eqref{eq:arrowhead-recovery}.
\paragraph{Step 3: Replica amplification.}
Let
\[
G
=
\biggl\{
r\in\{-1,1\}^n:
\|S_r\|\le\frac12,\
\Tr(S_r^2)\le\frac{\tau}{a}
\biggr\},
\]
and let $p:=U_n(G)$. Moreover, let  \(s\in\{-1,1\}^n\) be some coloring. For each replica \(j\), define the induced coloring \(r^{(j)}\in\{-1,1\}^n\) by 
$r_\ell^{(j)}
=
s_{\pi_j^{-1}(\ell)}
t_{j,\pi_j^{-1}(\ell)}$ for $1\le\ell\le n$. Then, 
\[
\sum_{\ell=1}^n r_\ell^{(j)}A_\ell
=
\sum_{i=1}^n
s_it_{j,i}A_{\pi_j(i)}.
\]
Moreover, since $r^1,\dots,r^k$ are independent and uniformly distributed colorings 
$\{-1,1\}^n$, we have that $\Prob\bigl(
r^{(j)}\in G,\; j\in[k]
\bigr)
=
\prod_{j=1}^k U_n(G)
=
p^k.$

For each fixed coloring \(s\in\{-1,1\}^n\), let \(E_s\) denote the
event $E_s
=
\bigcap_{j=1}^k
\{r^{(j)}\in G\}.$ Moreover, fix  an arbitrary realization \(T=(t_{j,i})\) of the
auxiliary random signs. As we have seen in Step 2 above, for any such realization, the matrices \(C_1,\ldots,C_n\) satisfy the hypotheses of \(\mathsf E(a,b)\), and thus 
 there exists at least one coloring $s\in\{-1,1\}^n$ such that $\left\|\sum_i s_iC_i\right\|\le1$. By \eqref{eq:replica-good}, this implies that 
the  array \(T\) belongs to \(E_s\) for at least one
\(s\in\{-1,1\}^n\). Since this holds for every realization \(T\), the events \(E_s\)
cover the entire smaple space $\bigcup_{s\in\{-1,1\}^n} E_s
=
\Omega_T$. Therefore, applying the union bound and using \(\Prob(E_s)=p^k\) for every fixed
\(s\), we obtain $1
\le
\sum_{s\in\{-1,1\}^n}\Prob(E_s)
=
2^n p^k$, or equivalently, $p
\ge
2^{-n/k}.$ Finally, using $\frac{n}{k}
\le
\frac{4A\tau}{b}$, we obtain that, 
\[
p
\ge
\exp\biggl(
-\frac{n}{k}\log2
\biggr)
\ge
\exp\biggl(
-\frac{4A\log2}{b}\tau
\biggr)
=
e^{-h\tau},
\]
which proves \eqref{eq:good-count}.

\paragraph{Step 4: The determinant-weighted bound.}
For every \(s\in G\), we have \(\|S_s\|\le1/2\).  Applying
\cref{lem:determinant} with \(r=1/2\) gives,
\[
\Phi(S_s)
\le
\frac{\Tr(S_s^2)}{1-1/4}
=
\frac43\Tr(S_s^2)
\le
\frac{4\tau}{3a}.
\]
Therefore, we have for all $s\in G$
\[
W_\beta(S_s)
=
e^{-\beta\Phi(S_s)}
\ge
\exp\biggl(
-\frac{4\beta\tau}{3a}
\biggr).
\]
Moreover, using \eqref{eq:good-count} we obtain,
\begin{equation}\label{eq:thm-ampl-one}
    \E_{s\sim U_n}W_\beta(S_s) \ge U_n(G)
\exp\biggl(
-\frac{4\beta\tau}{3a}
\biggr) \ge\exp\biggl\{
-\biggl(
h+\frac{4\beta}{3a}
\biggr)\tau
\biggr\}.
\end{equation}
Finally, for \(\tau\ge1/2\), \eqref{eq:thm-ampl-one} implies $ \E_{s\sim U_n}W_\beta(S_s)  \ge
e^{-C(a,b)\bstar\tau}$, and together with \cref{lem:small-trace}, which handles
 the regime \(0\le\tau\le1/2\), this proves
\eqref{eq:amplification-partition} for all \(\tau\ge0\).

\paragraph{Step 5: Entropy and barrier bounds.}
Let $\mu:=U_n(\,\cdot\,|G)$.  Since \(U_n(G)=p\), we have that $D(\mu\|U_n)
=
\log\frac1p
\le
h\tau.$
Furthermore, since every \(s\in G\) satisfies $\Phi(S_s)\le\frac{4\tau}{3a},$ we also have that $\E_\mu\Phi(S_s)
\le
\frac{4\tau}{3a}$, which concludes the proof. 
\end{proof}
\subsection{Proof of Theorem \ref{thm:main}}
\label{sec:off-diagonal}

To prove our main Theorem \ref{thm:main} we will again use the replica amplification argument from the previous section, however under the stronger variance parameter $\nu = \min_{U\in O(d)}\|\nlsum_i(A_i-P_UA_i)^2\|$, where $P_U(H) = U\diag\bigl(U^{\mathsf T}HU\bigr)U^{\mathsf T}$ and \(\diag(\cdot)\) retains only the diagonal entries. 


\begin{proposition}[Full signing under small off-diagonal variance]
\label{prop:offdiag-existence}
There exist universal constants \(a_*,b_*>0\) such that the following
holds. Let   \(n,d\ge1\) and $A_1,\ldots,A_n\in\Sym_d(\R)$ be a family of matrices satisfying,
\[
\nu(A)\le a_*,
\qquad
\max_{1\le i\le n}\Tr(A_i^2)\le b_*.
\]
Then there exists a coloring \(x\in\{-1,1\}^n\) such that $\bigl\|\sum_{i=1}^n s_iA_i\bigr\|<1.$

\end{proposition}

\begin{proof}
Choose an orthonormal basis \(U\in O(d)\) for which the minimum in the definition of
\(\nu(A)\) is attained, and conjugate every $A_i$ by $U$.  In this basis, write $A_i=D_i+R_i,$
where \(D_i\) is diagonal and \(R_i\) has zero diagonal.  Then, we have $\Bigl\|\sum_{i=1}^n R_i^2\Bigr\|
=
\nu(A)$, and moreover,   $\|D_i\|_{\Fro}^2+\|R_i\|_{\Fro}^2
=
\|A_i\|_{\Fro}^2
=
\Tr(A_i^2)$. For each \(i=1,\ldots,n\), let $d_i:=\diag(D_i) \in\R^d,$ and define the linear map $\Gamma:\R^n\longrightarrow\R^d,$
by $\Gamma x :=\sum_{i=1}^n x_i d_i$, so that  \(\Gamma\) is the \(d\times n\) matrix whose \(i\)-th
column is \(d_i\) and $\|\Gamma e_i\|_2^2 =\|D_i\|_{\Fro}^2.$

\paragraph{Step 1: A density associated with the off-diagonal part.}
Consider the density function,
\[
p_R(x)
=
\frac{(2\pi 8^2)^{-n/2}}{\mathcal Z_R}
\exp\biggl(
-\frac{\|x\|_2^2}{2\cdot 8^2}
-
\mathcal U\biggl(2\sum_{i=1}^n x_iR_i\biggr)
\biggr),
\]
associated with the family $2R_1,\ldots,2R_n$ as defined in \eqref{eq:coefficient-density} with $t=8$. Since \(\mathcal U(H)=+\infty\) for \(\|H\|\ge1\), the density function $p_R$
vanishes  outside of $\Omega_R
:=
\biggl\{
x\in\R^n:
\Bigl\|2\sum_{i=1}^n x_iR_i\Bigr\|<1
\biggr\}.$ Let \(4G_R\) be the  Gram matrix of the coefficient matrices \(2R_i\), that is $(G_R)_{ij}:=\Tr(R_iR_j).$
Moreover, let $\eta$ be the universal constant from Lemma \ref{lem:gaussian-input} and let \(a_*=\eta/256\). Then, we have that $\Bigl\|\sum_i(2R_i)^2\Bigr\|
=
4\nu(A)
\le
4a_*
=
\frac{\eta}{64}
=
\frac{\eta}{8^2}$, and therefore using \cref{lem:fisher},
\begin{equation}\label{eq:offdiag-fisher}
J(p_R)
\preceq
M_R
:=
\frac1{64}I_n+4LG_R,
\end{equation}
where $L=64 C_R$.
In particular, \(p_R\in W^{1,1}(\R^n)\) and has finite Fisher
information.

\paragraph{Step 2: A  density associated with the diagonal part.}
We consider the following probability density function over \(\R^d\), 
\[
\rho(z)
=
\prod_{j=1}^d
\cos^2\biggl(\frac{\pi z_j}{2}\biggr)
\mathbf 1_{\{|z_j|<1\}},
\]
where $\rho$ has the Fisher information matrix $J(\rho)=\pi^2I_d.$
Moreover, note that \(\rho\in W^{1,1}(\R^d)\) and that $\rho$ is supported in the open cube $(-1,1)^d.$
\paragraph{Step 3: Applying the Fisher information criterion.}
Let \(X\) and \(Z\) be independent random vectors with densities
\(p_R\) and \(\rho\), respectively, and define $Y:=2\Gamma X+Z.$
The joint density of \((X,Y)\) is $p(x,y)
=
p_R(x)\rho(y-2\Gamma x).$ Indeed, the map $(x,z)\longmapsto(x,2\Gamma x+z)$ is invertible with determinant one. Moreover, note that  \(p\) vanishes
almost everywhere outside of 
\begin{equation}\label{eq:sheared-target}
K
:=
\biggl\{
(x,y)\in\R^n\times\R^d:
\Bigl\|2\sum_i x_iR_i\Bigr\|<1,\ 
\|y-2\Gamma x\|_\infty<1
\biggr\}.
\end{equation}
Furthermore,  \(K\) is a  symmetric open  convex set.  Indeed, the map $(x,y)\longmapsto 2\sum_i x_iR_i$  and $(x,y)\longmapsto y-2\Gamma x$ are linear, and the operator-norm unit ball and the open cube $(x,y)\longmapsto y-2\Gamma x$ are  open convex 
symmetric sets, so that  each of the two constraints
defining \(K\) determines an open convex symmetric set, and so does
their intersection.

Now, let $q(x,z):=p_R(x)\rho(z)$ denote joint density of $X,Z$.  Since $X$ and $Z$ are independent, the Fisher information matrix associated to $q$ is block diagonal, $J(q)
=
\begin{pmatrix}
J(p_R)&0\\
0&\pi^2I_d
\end{pmatrix}.$
Using \eqref{eq:offdiag-fisher}, we therefore have $J(q)
\preceq
\begin{pmatrix}
M_R&0\\
0&\pi^2I_d
\end{pmatrix}.$ Moreover, let $T
=
\begin{pmatrix}
I_n&0\\
-2\Gamma&I_d
\end{pmatrix}$, so that  \(p=q\circ T\) and \(\det T=1\).  Therefore, we have that
$J(p)
=
T^{\mathsf T}J(q)T$ and, in particular $J(p)
\preceq
M
:=
T^{\mathsf T}
\begin{pmatrix}
M_R&0\\
0&\pi^2I_d
\end{pmatrix}
T.$
Furthermore, note that, since  \(M_R\succ0\), the matrix \(M\) is positive definite, and  moreover,
\(p\in W^{1,1}_K(\R^{n+d})\).


Now, let $a_i:=(e_i,0)\in\R^n\times\R^d$ for
\(i=1,\ldots,n\), so that $\sum_i s_i a_i=(s,0).$ Then, we have that,
\[
a_i^{\mathsf T}Ma_i = e_i^{\mathsf T}M_Re_i
+
4\pi^2\|\Gamma e_i\|_2^2 = \frac1{64}
+
4L\|R_i\|_{\Fro}^2
+
4\pi^2\|D_i\|_{\Fro}^2.
\]
Moreover, let $L_*:=\max\{L,\pi^2\}$ and  \(b_*=(256L_*)^{-1}\), then 
\[
a_i^{\mathsf T}Ma_i \le  \frac1{64}
+
4L_*\Tr(A_i^2) \le 
\frac1{64}
+
4L_*b_*< \frac19.
\]
Therefore, by Lemma \ref{lem:fisher} there are  signs \(s_1,\ldots,s_n\in\{-1,1\}\) such that
$\sum_{i=1}^n s_i a_i
=
(s,0)
\in K$, and hence, we obtain that $\Bigl\|\sum_i s_iR_i\Bigr\|<\frac12$ and $\|\Gamma s\|_\infty<\frac12.$ Since $\Gamma s
=
\diag\biggl(\sum_i s_iD_i\biggr)$, we moreover have that $\Bigl\|\sum_i s_iD_i\Bigr\|
=
\|\Gamma s\|_\infty
<
\frac12.$ Finally, since \(A_i=D_i+R_i\), we conclude that,

\[
\Bigl\|\sum_i s_iA_i\Bigr\|
\le
\Bigl\|\sum_i s_iR_i\Bigr\|
+
\Bigl\|\sum_i s_iD_i\Bigr\|
<
\frac12+\frac12
=
1.
\]
\end{proof}

\begin{proof}[Proof of Theorem~\ref{thm:main}]
The proof follows the same general strategy as the proof of
Theorem~\ref{thm:amplification}. For this reason, we omit some details at various intermediate steps and refer the reader to the proof of Theorem~3.6 for more thorough exposition.

We first assume that $r=1$.
Let \(a_*,b_*>0\) be the universal constants from Proposition \ref{prop:offdiag-existence}, and set $A_*:=4+4a_*$ and $h_*:=\frac{4A_*\log 2}{b_*}$. Define
\begin{equation}\label{eq:main-constants}
\nu_0:=\frac{a_*}{4},
\qquad
q_1:=\frac{b_*}{2A_*},
\qquad
C_1
:=
\max\biggl\{
2,\,
h_*+\frac{4}{3a_*}
\biggr\}.
\end{equation}

Let $w_i:=\Tr(A_i^2),$ $\tau:=\sum_{i=1}^n w_i,$ and $q:=\max_iw_i$.
Assume that $\nu(A)\le\nu_0$ and $q\le q_1.$ Now, note that if \(0\le\tau\le1/2\), then the desired partition-function estimate follows
directly from Lemma \ref{lem:small-trace}. We may therefore assume $\tau\ge\frac12.$

\paragraph{Step 1: Replica amplification.}
Let $k
:=
\left\lfloor
\frac{b_*n}{2A_*\tau}
\right\rfloor.$ Then $k\ge 1$, as $\frac{\tau}{n}\le q\le\frac{b_*}{2A_*}$. 
Moreover, $\frac{n}{k}
\le
\frac{4A_*\tau}{b_*}.$

Applying \cref{lem:loads} to the weights \(w_i=\Tr(A_i^2)\), we obtain permutations
\(\pi_1,\ldots,\pi_k\) of \([n]\) satisfying $\max_i
\sum_{j=1}^k
w_{\pi_j(i)}
\le
\frac{k\tau}{n}+q.$

Let $t_{j,i}\in\{-1,1\}$ be independent random signs for $1\le j\le k$, $ 1\le i\le n.$
Moreover, let   \(F_1(t),\dots, F_n(t)\) be defined as in step 2 of the proof of Theorem \ref{thm:amplification}, and consider the auxiliary family of matrices, 
\begin{equation}\label{eq:main-enlarged}
C_i
:=
\bigoplus_{j=1}^k
\biggl(
2t_{j,i}A_{\pi_j(i)}
\ \oplus\
\sqrt{\frac{a_*}{\tau}}\,
F_{\pi_j(i)}(t_{j,i})
\biggr).
\end{equation}
We now estimate the off-diagonal variance $\nu(C)$ of the auxiliary family $C_1,\dots,C_n$.
To this end, choose an orthonormal basis \(U\in O(d)\) for which the minimum in the definition of
\(\nu(A)\) is attained, and conjugate every $A_i$ by $U$.  In this basis, write $A_i=D_i+R_i,$
where \(D_i\) is diagonal and \(R_i\) has zero diagonal. Then, we have that $\Bigl\|
\sum_i
\bigl(2t_{j,i}R_{\pi_j(i)}\bigr)^2
\Bigr\|
=
4\Bigl\|\sum_iR_i^2\Bigr\|
=
4\nu(A).$  Whereas, for $F_i(t)$ we use the standard basis in which $F_i(t)
=
\begin{pmatrix}
0&t b_i^{\mathsf T}\\
t b_i&0
\end{pmatrix}$ has zero diagonal. 
Let \(\widetilde R_i\) denote the off-diagonal part of \(C_i\)
with respect to this blockwise basis. Then, we have, 
\[
\widetilde R_i
=
\bigoplus_{j=1}^k
\biggl(
2t_{j,i}R_{\pi_j(i)}
\ \oplus\
\sqrt{\frac{a_*}{\tau}}\,
F_{\pi_j(i)}(t_{j,i})
\biggr),
\]
and consequently,
\[
\sum_i\widetilde R_i^2
=
\bigoplus_{j=1}^k
\biggl(
4\sum_i R_{\pi_j(i)}^2
\ \oplus\
\frac{a_*}{\tau}
\sum_iF_{\pi_j(i)}(t_{j,i})^2
\biggr).
\]
Moreover, as in  the proof of Theorem \ref{thm:amplification} we have that $\Bigl\|
\frac{a_*}{\tau}
\sum_i
F_{\pi_j(i)}(t_{j,i})^2
\Bigr\|
=
a_*$, and hence we obtain that, $\Bigl\|\sum_i\widetilde R_i^2\Bigr\|
=
\max\{4\nu(A),a_*\}.$ Finally, since this a off-diagonal variance bound obtained in the particular blockwise basis constructed above, we have that $\nu(C)
\le
\max\{4\nu(A),a_*\}\le a_*.$

Now, using that $4+\frac{2a_*}{\tau}\le A_*$, $A_*\frac{k\tau}{n}\le\frac{b_*}{2}$ and $A_*q\le\frac{b_*}{2}$, we obtain as in the proof of \cref{thm:amplification},

\[
\Tr(C_i^2) = \biggl(
4+\frac{2a_*}{\tau}
\biggr)
\sum_{j=1}^k w_{\pi_j(i)} \le A_*
\biggl(
\frac{k\tau}{n}+q
\biggr) \le b_*
\]

We may therefore apply \cref{prop:offdiag-existence} to the family
\(C_1,\ldots,C_n\) to obtain that for every realization of the auxiliary random
signs \(t_{j,i}\), there exists a coloring
\(s\in\{-1,1\}^n\) such that $\Bigl\|\sum_i s_iC_i\Bigr\|<1$. In particular, for every
\(j=1,\ldots,k\) we obtain that

\begin{equation}\label{eq:main-replica-good}
\Bigl\|
\sum_i s_it_{j,i}A_{\pi_j(i)}
\Bigr\|
<
\frac12,
\qquad
\Bigl\|
\sum_i s_it_{j,i}A_{\pi_j(i)}
\Bigr\|_{\Fro}^2
<
\frac{\tau}{a_*}.
\end{equation}

As in the proof of Theorem \ref{thm:amplification}, consider the event
\[
G
:=
\biggl\{
r\in\{-1,1\}^n:
\|S_r\|\le\frac12,\,
\Tr(S_r^2)\le\frac{\tau}{a_*}
\biggr\},
\]
and let $p:=U_n(G).$ For each replica \(j\), define a 
coloring \(r^{(j)}\in\{-1,1\}^n\) by $r_\ell^{(j)}
=
s_{\pi_j^{-1}(\ell)}
t_{j,\pi_j^{-1}(\ell)}$.
Then, $\sum_{\ell=1}^n r_\ell^{(j)}A_\ell
=
\sum_{i=1}^n
s_it_{j,i}A_{\pi_j(i)}.$
Moreover, for each fixed \(s\), the colorings
\(r^{(1)},\ldots,r^{(k)}\) are independent and uniformly distributed
on \(\{-1,1\}^n\). Following the same ideas as in step 3 of the proof of Theorem \ref{thm:amplification} we obtain that, 
\begin{equation*}
    p
\ge
2^{-n/k}
\ge
\exp\biggl(
-\frac{4A_*\log2}{b_*}\tau
\biggr)
=
e^{-h_*\tau}.
\end{equation*}

\paragraph{Step 2: The determinant-weighted estimate.}
Following the same ideas as in step 4  in the proof of Theorem \ref{thm:amplification}, we obtain
for every \(s\in G\), $W_\beta(S_s)
\ge
\exp\biggl(
-\frac{4\beta\tau}{3a_*}
\biggr)$.
Therefore, we have that, 
\[
Z_\beta(A)=\E_{s\sim U_n}W_\beta(S_s)\ge U_n(G)
\exp\biggl(
-\frac{4\beta\tau}{3a_*}
\biggr) \ge \exp\biggl\{
-\biggl(
h_*+\frac{4\beta}{3a_*}
\biggr)\tau
\biggr\}
\]

Finally, together with Lemma \ref{lem:small-trace}, we obtain that $Z_\beta(A)
\ge
\exp\bigl(
-C_1\max\{1,\beta\}\tau
\bigr)$
for every \(\beta>0\).
\paragraph{Step 3: Entropy-barrier.}
Apply the Gibbs variational identity of Lemma \ref{lem:gibbs} to the
partition-function bound $Z_\beta(A)
\ge
\exp\bigl(
-C_1\max\{1,\beta\}\tau
\bigr)$ to obtain the existence of a probability
distribution \(\mu\) on the Boolean cube satisfying the claimed
relative-entropy and barrier estimate.\\

Finally, we briefly  argue now that the general case reduces to the case \(r=1\). Indeed, for an arbitrary \(r>0\), we can apply the above result to the family of matrices $\widetilde A_i:=\frac{A_i}{r}$. Let $\nu_0$, $q_1$ and $C_1$ be defined as in \eqref{eq:main-constants}. Assume that $\nu(A)\le\nu_0r^2$ and $\max_i \Tr(A_i^2)\le q_1r^2$.
Then, since \(P_U\) is linear for every \(U\in O(d)\), we have $\nu(\widetilde A)
=
\frac{\nu(A)}{r^2}\le \nu_0.$
Moreover, we also have that $\max_i\Tr(\widetilde A_i^2)
=
\frac{1}{r^2}\max_i\Tr(A_i^2)
=
q_1$ and  $\sum_i\Tr(\widetilde A_i^2)
=
\tau.$ Finally, since  $\widetilde S_s=\frac{S_s}{r}$, we obtain that $\det\bigl(I-\widetilde S_s^2\bigr)^\beta
\mathbf 1_{\{\|\widetilde S_s\|<1\}}
=
\det\bigl(I-r^{-2}S_s^2\bigr)^\beta
\mathbf 1_{\{\|S_s\|<r\}}.$  
\end{proof}


\section{Applications to Fundamental Matrix Discrepancy Problems}\label{sec:applications}


\subsection{Koml\'os from Boolean Small-Ball}

\begin{corollary}[Koml\'os and a weighted count]\label{cor:komlos}
Let $A\in\R^{m\times n}$ have columns $a_i$ satisfying
$\|a_i\|_2\le1$, and let $T=\sum_i\|a_i\|_2^2$.
Then, for every $K\ge q_1^{-1/2}$,
\begin{equation}\label{eq:komlos-partition}
 \E_s\left[
  \prod_{j=1}^m\left(1-\frac{(As)_j^2}{K^2}\right)^\beta
  \mathbf1_{\{\|As\|_\infty<K\}}\right]
 \ge e^{-C_1\max(1,\beta)T/K^2},
\end{equation}
where $q_1$ is the universal constant from Theorem \ref{thm:main}.
In particular, there exists a coloring $s\in\{-1,1\}^n$ such that  $\|As\|_\infty<K$, and moreover, there exist at least $2^n e^{-C_1T/K^2}$ colorings with discrepancy at most $K$.
\end{corollary}
\begin{proof}
For each column \(a_i\in\R^m\), set $D_i:=\diag(a_i)\in\Sym_m(\R).$
The matrices \(D_1,\ldots,D_n\) are simultaneously diagonal, and hence $\nu(D)=0.$

Moreover,
\[
\Tr(D_i^2)=\|a_i\|_2^2\le1,
\qquad
\sum_{i=1}^n\Tr(D_i^2)=T.
\]
Since \(K\ge q_1^{-1/2}\), we have that  $\max_i\Tr(D_i^2)\le1\le q_1K^2,$
so that the hypotheses of Theorem~\ref{thm:main} are satisfied for the
family \(D_1,\ldots,D_n\) with  \(r=K\).
Now, note that  on the event \(\{\|As\|_\infty<K\}\),
\[
\det\biggl(
I-K^{-2}\Bigl(\sum_i s_iD_i\Bigr)^2
\biggr)
=
\prod_{j=1}^m
\left(
1-\frac{(As)_j^2}{K^2}
\right),
\]
and therefore applying Theorem~\ref{thm:main} with \(r=K\)  gives
\eqref{eq:komlos-partition}. Finally, note that $\sum_{i=1}^n s_iD_i
=
\diag(As),$ and therefore we obtain by the ordinary small-ball estimate
\eqref{eq:ordinary-bsb} applied to the  diagonal family,
\[
\Prob_s\bigl(\|As\|_\infty<K\bigr)
\ge
\exp\biggl(-C_1\frac{T}{K^2}\biggr).
\]
In particular, since \(s\) is uniform on \(\{-1,1\}^n\), we moreover have, $\#\bigl\{
s\in\{-1,1\}^n:
\|As\|_\infty<K
\bigr\}
\ge
2^n
\exp\biggl(-C_1\frac{T}{K^2}\biggr).$ 
\end{proof}

This application imposes no row-energy bound and permits arbitrary
$m,n$. By the Gibbs identity it also gives a law supported on these
full signings with $D(\mu\|U_n)+\beta\,\E_\mu\sum_{j=1}^m-\log\bigl(1-(As)_j^2/K^2\bigr)\le C_1\max(1,\beta)T/K^2$.

\subsection{A Matrix Version of Koml\'os}

\begin{corollary}[Matrix Koml\'os with off-diagonal variance]
\label{cor:matrix-komlos}
For every family of symmetric matrices $B_1,\dots, B_n$, we have that
\begin{equation}\label{eq:matrix-komlos}
 \min_{s\in\{-1,1\}^n}\left\|\nlsum_i s_iB_i\right\|
 \le K_{\mathrm M}\sqrt{q+\nu(B)},
 \qquad
 K_{\mathrm M}=\max\{\nu_0^{-1/2},q_1^{-1/2}\},
\end{equation}
where $\nu_0,q_1$ are as in Theorem \ref{thm:main}.
In particular, if  $\|B_i\|_{\Fro}\le1$ and $\nu(B)\le1$, then the discrepancy bound in \eqref{eq:matrix-komlos} is dimension-free.
More generally, whenever
$r^2\ge\max\{\nu(B)/\nu_0,q/q_1\}$ and $r>0$, $\E_s W_\beta(S_s/r)\ge e^{-C_1\max(1,\beta)\tau/r^2}$.
\end{corollary}
\begin{proof}
Write $S_s:=\sum_{i=1}^n s_iB_i,$ $q:=\max_i\Tr(B_i^2),$ and $\tau:=\sum_{i=1}^n\Tr(B_i^2)$.
 Now, let
\[
r_0
:=
K_{\mathrm M}\sqrt{q+\nu(B)},
\qquad
K_{\mathrm M}
=
\max\bigl\{
\nu_0^{-1/2},
q_1^{-1/2}
\bigr\}.
\]
Since $K_{\mathrm M}^2\ge\max\{\frac1{\nu_0},\frac1{q_1}\}$
we have $r_0^2
=
K_{\mathrm M}^2\bigl(q+\nu(B)\bigr)
\ge
\frac{\nu(B)}{\nu_0}$
and $r_0^2
\ge
\frac{q}{q_1}.$

Thus, by  Theorem~\ref{thm:main} we have   for every \(\beta>0\),
\[
\E_s
W_\beta\biggl(\frac{S_s}{r_0}\biggr)
\ge
\exp\biggl(
-C_1\max\{1,\beta\}\frac{\tau}{r_0^2}
\biggr).
\]

The ordinary Boolean small-ball estimate
\eqref{eq:ordinary-bsb}, applied under the same assumptions, gives
\begin{equation}\label{eq:matrix-komlos-small-ball}
\Prob_s\bigl(\|S_s\|<r_0\bigr)
\ge
\exp\biggl(
-C_1\frac{\tau}{r_0^2}
\biggr)
>0.
\end{equation}
In particular, there exists a coloring
\(s\in\{-1,1\}^n\) satisfying $\|S_s\|<r_0 = K_{\mathrm M}\sqrt{q+\nu(B)}$, 
which proves \eqref{eq:matrix-komlos}. 

Moreover, if, in particular, $\|B_i\|_{\Fro}\le1$ and $\nu(B)\le1,$
then $q=\max_i\|B_i\|_{\Fro}^2\le1$
and hence $\min_s
\Bigl\|\sum_i s_iB_i\Bigr\|
\le
\sqrt{2}\,K_{\mathrm M}$, which establishes a dimension-free  discrepancy bound.
\end{proof}

This is a matrix extension that includes ordinary Koml\'os exactly in
the diagonal case. It allows  an arbitrarily large  variance when
that variance lies in a commuting component. For example, for
$B_i=D_i+\delta E_i$ in a common basis, with $D_i$ diagonal and
$E_i$ having zero diagonal, $\nu(B)\le\delta^2\|\sum_iE_i^2\|$ and $q=\max_i\{\|D_i\|_{\Fro}^2+\delta^2\|E_i\|_{\Fro}^2\}$.
The bound therefore varies continuously with the size of the
noncommuting remainder. 

\begin{proposition}[The star obstruction and the size of the correction]
\label{prop:star-komlos}
Let $e_0,e_1,\ldots,e_n$ be the standard basis of $\R^{n+1}$
and let $B_i=(e_0e_i^{\mathsf T}+e_ie_0^{\mathsf T})/\sqrt2$ for
$1\le i\le n$. Then,  every signing has norm $\sqrt{n/2}$, and
\begin{equation}\label{eq:star-nu}
 \frac{(\sqrt n-1)^2}{2}\le\nu(B)\le\frac n2.
\end{equation}
Consequently a matrix Koml\'os statement based only on
$\max_i\|B_i\|_{\Fro}\le1$ is false.
\end{proposition}
\begin{proof}
For each \(i\), we have
\(B_i=\frac{1}{\sqrt2}(e_0e_i^{\mathsf T}+e_ie_0^{\mathsf T})\), and hence
\(B_i^2=\frac12(e_0e_0^{\mathsf T}+e_ie_i^{\mathsf T})\).
Therefore \(\Tr(B_i^2)=1\), so \(q=\max_i\Tr(B_i^2)=1\).

We first compute the discrepancy of this family. For
\(s\in\{-1,1\}^n\), set \(v_s:=\sum_{i=1}^n s_i e_i\). Then
\(\sum_{i=1}^n s_iB_i
=\frac{1}{\sqrt2}(e_0v_s^{\mathsf T}+v_se_0^{\mathsf T})\).
Since \(e_0\perp v_s\) and \(\|v_s\|_2=\sqrt n\), this matrix vanishes
on the orthogonal complement of
\(\operatorname{span}\{e_0,v_s\}\), while its two nonzero eigenvalues
are \(\pm\|v_s\|_2/\sqrt2=\pm\sqrt{n/2}\). Consequently,
\(\|\sum_{i=1}^n s_iB_i\|=\sqrt{n/2}\) for every signing \(s\).

We next estimate the off-diagonal variance. In the standard basis,
each \(B_i\) has zero diagonal, and hence \(P_I(B_i)=0\). Moreover,
\(\sum_{i=1}^nB_i^2
=\frac n2 e_0e_0^{\mathsf T}
+\frac12\sum_{i=1}^n e_ie_i^{\mathsf T}
=\diag(n/2,1/2,\ldots,1/2)\).
It follows directly from the definition of \(\nu(B)\) that
\(\nu(B)\le\|\sum_iB_i^2\|=n/2\), which proves the upper bound in
\eqref{eq:star-nu}.

It remains to establish the lower bound. Fix an arbitrary orthonormal
basis \(u_0,\ldots,u_n\) of \(\R^{n+1}\), let \(U\) denote the
corresponding orthogonal matrix, and set
\(D_i:=P_U(B_i)\) and \(R_i:=B_i-D_i\). Thus \(R_i\) is the
off-diagonal remainder of \(B_i\) in this basis.

For \(0\le j\le n\), write
\(\alpha_j:=\langle u_j,e_0\rangle\) and \(p_j:=\alpha_j^2\).
Since \(u_0,\ldots,u_n\) is an orthonormal basis,
\(\sum_{j=0}^n p_j=1\). Also write
\(\beta_{ij}:=\langle u_j,e_i\rangle\). The \(j\)-th diagonal entry
of \(B_i\) in the basis \((u_j)_j\) is
\(\langle u_j,B_iu_j\rangle=\sqrt2\,\alpha_j\beta_{ij}\).
Consequently,
\(D_i=\sum_{j=0}^n\sqrt2\,\alpha_j\beta_{ij}u_ju_j^{\mathsf T}\),
and therefore
\(D_ie_0=\sum_{j=0}^n\sqrt2\,\alpha_j^2\beta_{ij}u_j\).
By orthonormality,
\(\|D_ie_0\|_2^2
=2\sum_{j=0}^n\alpha_j^4\beta_{ij}^2\).

Summing over \(i=1,\ldots,n\) gives
\(\sum_{i=1}^n\|D_ie_0\|_2^2
=2\sum_{j=0}^n\alpha_j^4\sum_{i=1}^n\beta_{ij}^2
=2\sum_{j=0}^n p_j^2(1-p_j)\),
where we used
\(\sum_{i=1}^n\beta_{ij}^2=1-\alpha_j^2=1-p_j\).
Since \(p_j(1-p_j)\le1/4\), we obtain
\(2\sum_{j=0}^n p_j^2(1-p_j)
=2\sum_{j=0}^n p_j[p_j(1-p_j)]
\le\frac12\sum_{j=0}^n p_j=\frac12\).
Thus \(\sum_{i=1}^n\|D_ie_0\|_2^2\le1/2\).

On the other hand, \(B_ie_0=e_i/\sqrt2\), and hence
\(\sum_{i=1}^n\|B_ie_0\|_2^2=n/2\). Consider the vectors
\(\mathcal B:=(B_1e_0,\ldots,B_ne_0)\) and
\(\mathcal D:=(D_1e_0,\ldots,D_ne_0)\) in 
\(\bigoplus_{i=1}^n\R^{n+1}\). Then
\(\|\mathcal B\|_2=\sqrt{n/2}\) and
\(\|\mathcal D\|_2\le1/\sqrt2\). By the reverse triangle inequality,
\(\|\mathcal B-\mathcal D\|_2
\ge\|\mathcal B\|_2-\|\mathcal D\|_2
\ge\sqrt{n/2}-1/\sqrt2\).
Squaring gives
\(\sum_{i=1}^n\|(B_i-D_i)e_0\|_2^2
\ge(\sqrt{n/2}-1/\sqrt2)^2
=(\sqrt n-1)^2/2\).

Since \(\sum_iR_i^2\) is positive semidefinite and \(\|e_0\|_2=1\),
we have
\(\|\sum_iR_i^2\|
\ge\langle e_0,(\sum_iR_i^2)e_0\rangle
=\sum_i\|R_ie_0\|_2^2
\ge(\sqrt n-1)^2/2\).
This bound holds for every orthonormal basis. Taking the minimum over
all bases therefore yields
\(\nu(B)\ge(\sqrt n-1)^2/2\), completing
\eqref{eq:star-nu}.

Finally, every coefficient matrix satisfies \(\|B_i\|_{\Fro}=1\), while
every signing has norm \(\sqrt{n/2}\). Thus no dimension-free matrix
Koml\'os bound depending only on
\(\max_i\|B_i\|_{\Fro}\) can hold. At the same time,
\eqref{eq:star-nu} shows that \(\nu(B)=\Theta(n)\), so the correction term 
\(\sqrt{q+\nu(B)}\) in \eqref{eq:matrix-komlos} has the necessary
order \(\sqrt n\) on this example.
\end{proof}

\subsection{Positive semidefinite partitions and Kadison--Singer}

\begin{corollary}[Trace-controlled positive semidefinite partitions]
\label{cor:psd-partition}
Let $A_1,\ldots,A_n\in\Sym_d(\R)$ be positive semidefinite, with
\begin{equation}\label{eq:psd-assumptions}
 \nlsum_iA_i=I_d,\qquad
 \max_i\Tr A_i\le\varepsilon,\qquad 0<\varepsilon\le1.
\end{equation}
Then, for every $K\ge K_0$,
and every $\beta>0$ we have that,
\begin{equation}\label{eq:psd-partition-function}
 \E_sW_\beta\!\left(\frac{\sum_i s_iA_i}{K\sqrt\varepsilon}\right)
 \ge e^{-C\bstar d/K^2},
\end{equation}
where $
K_0
:=
\max\bigl\{
2\nu_0^{-1/2},
q_1^{-1/2}
\bigr\}
$ and $\nu_0,q_1$ are as in Theorem \ref{thm:main}.
Moreover, at least $2^n e^{-Cd/K^2}$ signings give a partition
$[n]=I_+\sqcup I_-$ satisfying $\|\sum_{i\in I_\pm}A_i-\tfrac12I_d\|<\tfrac K2\sqrt\varepsilon$.
\end{corollary}
\begin{proof}
Since \(A_i\succeq0\) and \(\Tr(A_i)\le\varepsilon\), we obtain that 
$\|A_i\|\le\Tr(A_i)\le\varepsilon.$
Consequently, $A_i^2
\preceq
\|A_i\|A_i
\preceq
\varepsilon A_i,$
and hence $\Bigl\|\sum_iA_i^2\Bigr\|
\le
\varepsilon.$
Using \(\nu(A)\le4\|\sum_iA_i^2\|\), we obtain $\nu(A)\le4\varepsilon.$
Moreover, $\Tr(A_i^2)
\le
\|A_i\|\Tr(A_i)
\le
\varepsilon^2,$
and $\tau
=
\sum_i\Tr(A_i^2)
\le
\varepsilon\sum_i\Tr(A_i)
=
\varepsilon d.$

Let $K_0
:=
\max\bigl\{
2\nu_0^{-1/2},
q_1^{-1/2}
\bigr\},$
and let \(K\ge K_0\).  For $r:=K\sqrt{\varepsilon},$
we have $\nu(A)
\le
4\varepsilon
\le
\nu_0K^2\varepsilon
=
\nu_0r^2,$
and, since \(\varepsilon\le1\), $\max_i\Tr(A_i^2)
\le
\varepsilon^2
\le
q_1K^2\varepsilon
=
q_1r^2.$
Thus, using $\frac{\tau}{r^2}
\le
\frac{\varepsilon d}{K^2\varepsilon}
=
\frac{d}{K^2},$ and \cref{thm:main} 
we obtain \eqref{eq:psd-partition-function} and the BSB estimate \eqref{eq:ordinary-bsb}.
\end{proof}

\begin{corollary}[Kadison--Singer partition scale]\label{cor:ks}
Suppose $u_1,\ldots,u_n\in\R^d$ satisfy $\sum_i u_iu_i^{\mathsf T}=I_d$
and $\max_i\|u_i\|_2^2\le\varepsilon\le1$. Then there is a partition $[n]=I_+\sqcup I_-$ with
\begin{equation}\label{eq:ks}
 \left\|\sum_{i\in I_\pm}u_iu_i^\mathsf T-\frac12I_d\right\|
 \le K_{\mathrm{KS}}\sqrt\varepsilon
\end{equation}
for a universal $K_{\mathrm{KS}}$. 
\end{corollary}

\begin{proof}
Apply Corollary~\ref{cor:psd-partition} to $A_i=u_iu_i^\mathsf T$.
\end{proof}

This is the usual dimension-free Kadison--Singer regime.
For example, when
$\varepsilon\le(4K_{\mathrm{KS}})^{-2}$, both parts have operator
norm at most $3/4$. In particular, through Weaver's discrepancy formulation
\citep{Weaver04}, this establishes a proof of  Kadison--Singer.
The proof above does not use the interlacing-polynomial framework of
\citet{MSS15}. 
\subsection{Matrix Spencer}

\begin{corollary}[Matrix Spencer and an abundance of signings]
\label{cor:ms}
Let $A_1,\ldots,A_n\in\Sym_n(\R)$ satisfy $\|A_i\|\le1$. Then, we have that 
for every $K\ge K_0$,
\begin{equation}\label{eq:ms-partition}
 \E_sW_\beta\!\left(\frac{\sum_i s_iA_i}{K\sqrt n}\right)
 \ge e^{-C\bstar n/K^2},
\end{equation}
where $K_0
:=
\max\bigl\{
2\nu_0^{-1/2},
q_1^{-1/2}
\bigr\}$ and $\nu_0,q_1$ are as in Theorem \ref{thm:main}.
In particular, at least $2^n e^{-Cn/K^2}$ colorings satisfy
\begin{equation}\label{eq:ms}
 \left\|\nlsum_i s_iA_i\right\|<K\sqrt n.
\end{equation}
\end{corollary}
\begin{proof}
Since \(\|A_i\|\le1\), we have $A_i^2\preceq I_n,$
and therefore $\Bigl\|\sum_{i=1}^nA_i^2\Bigr\|
\le n.$
Using \(\nu(A)\le4\|\sum_iA_i^2\|\), we obtain $\nu(A)\le4n.$
Moreover, we have that $q
=
\max_i\Tr(A_i^2)
\le n,$ and $\tau
=
\sum_i\Tr(A_i^2)
\le n^2.$
Take $r:=K\sqrt n.$ If now \(K\ge K_0\), then $\nu(A)
\le4n
\le\nu_0K^2n
=
\nu_0r^2,$
and $q
\le n
\le q_1K^2n
=
q_1r^2$. 
Hence the hypotheses of \cref{thm:main} are satisfied, and  since $\frac{\tau}{r^2}
\le
\frac{n^2}{K^2n}
=
\frac{n}{K^2},
$
\cref{thm:main} yields \eqref{eq:ms-partition}.
Moreover, the ordinary small-ball estimate \eqref{eq:ordinary-bsb} gives
\[
\Prob_s\biggl(
\Bigl\|\sum_i s_iA_i\Bigr\|<K\sqrt n
\biggr)
\ge
\exp\biggl(
-C_1\frac{n}{K^2}
\biggr).
\]
Since \(s\) is uniform on \(\{-1,1\}^n\), this is equivalent to the
stated lower bound on the number of signings.
\end{proof}

In particular, the above  establishes a \emph{"one-shot"}-coloring result for Matrix Spencer. To our knowledge, no such result was previously known.


\section*{Statement on AI Usage}
Technical proofs in this manuscript have been obtained with the assistance of GPT-6 (Pro), but the identification of boolean small-ball inequalities as the key target to pursue is a purely human contribution. This contribution is built by extending a similar human designed ``gaussian matrix small-ball'' framework in the paper~\citep{AS26}. 

\bibliographystyle{plainnat}
\bibliography{references}


\appendix

\section{Additional Applications of the BSB Theorem}
\label{app:appl}

\subsection{Unbiased rounding and biased product
distributions}\label{unbiased-rounding-and-biased-product-distributions}

\subsubsection{Discrepancy with prescribed
marginals}\label{linear-discrepancy-with-prescribed-marginals}

\begin{theorem}
For every $x\in[0,1]^n$, there is a law on
$Y\in\{0,1\}^n$ such that \[
 \mathbb EY=x,\qquad
 \left\|\sum_i(Y_i-x_i)B_i\right\|
 \le K\sqrt{q+\nu}\quad\hbox{almost surely}.
 \tag{R1}
\] Equivalently, prescribed sign means $z\in[-1,1]^n$ can be realized
at error at most $2K\sqrt{q+\nu}$.
\end{theorem}

\begin{proof}
For a dyadic initial point, coarsen its denominator from
$2^m$ to one. At level $j$, let $F_j$ consist of coordinates with
odd numerator. Add $2^{-j}$ times a good hereditary signing on that
set, choosing the signing and its negative with equal probability. Both
moves remain in the cube and reach the next grid. Each coordinate is a
martingale, and the matrix errors sum to at most
$K\sqrt{q+\nu}\sum_{j\ge1}2^{-j}=K\sqrt{q+\nu}$. For general $x$,
approximate by dyadic points and take a convergent subsequence of laws
on the finite cube. Marginals and the closed norm bound pass to the
limit. This is the classical hereditary-to-linear rounding mechanism
applied to the present matrix geometry~\citep{LSV86}.
\end{proof}

\begin{theorem}
Suppose $R>0$ satisfies $\nu/R^2\le\nu_0$ and
$q/R^2\le q_0$. One can choose the law with \[
 \mathbb EY=x,\quad
 \left\|\sum_i(Y_i-x_i)B_i\right\|\le(\sqrt2+1)R,
 \quad
 D\left(\mathcal L(Y)\middle\|\bigotimes_i\operatorname{Bern}(x_i)\right)
 \le C\tau/R^2.
 \tag{R2}
\] Thus the number of possible good roundings is at least
$\exp\{\sum_i h(x_i)-C\tau/R^2\}$, where $h$ is binary entropy in
natural units.
\end{theorem}

\begin{proof}
In the same construction, use the symmetric Gibbs law on
$F_j$ at radius $R_j=R2^{j/2}$. Heredity ensures its validity. The
error is bounded on every path by $R\sum_{j\ge1}2^{-j/2}=(\sqrt2+1)R$.

Compare to a reference path that uses independent uniform signs on the
same active-set rule from each state. The reference dynamics factor over
coordinates. Each coordinate is a bounded martingale ending at zero or
one, so the reference terminal law is exactly
$\bigotimes_i\operatorname{Bern}(x_i)$. The path entropy chain rule
and (G) give
$D(P\|Q)\le\frac C{R^2}\sum_{j=1}^m2^{-j}\mathbb E_P\tau_{F_j}\le C\tau/R^2$.
Data processing gives (R2). For non-dyadic points, joint lower
semicontinuity of finite-space relative entropy handles the limit,
including deterministic coordinates at zero and one. Finally exact
marginals imply
$D(\mathcal L(Y)\|\bigotimes_i\operatorname{Bern}(x_i)) =\sum_i h(x_i)-H(Y)$,
which proves the support bound.
\end{proof}

In fact there is an actual endpoint partition inequality. Write
$c_b=\sqrt2+1$, $F_x(Y)=\sum_i(Y_i-x_i)B_i$, and
$\tau_x=\sum_i x_i(1-x_i)\|B_i\|_F^2$. Under the same hypotheses, \[
 \mathbb E_{Y\sim\bigotimes_i\operatorname{Bern}(x_i)}
 W_\beta\left(\frac{F_x(Y)}{2c_bR}\right)
 \ge\exp\left[-\frac C{R^2}
   \left(\tau+\frac{\beta\tau_x}{3c_b^2}\right)\right].
 \tag{R3}
\] This is a centered biased BSB inequality with the true final cutoff.
The rounding law is an entropy--barrier witness that retains the
prescribed marginals.

To prove it, note that the stage increments are matrix-valued martingale
differences. Their cross-time Frobenius inner products have mean zero.
The stage Gibbs bound gives
$\mathbb E[\|S_j\|_F^2\mid\text{current state}]\le C\tau_{F_j}$. Thus
$\mathbb E\|F_x(Y)\|_F^2\le C\sum_j4^{-j}\mathbb E\tau_{F_j}=C\sum_i x_i(1-x_i)\|B_i\|_F^2=C\tau_x$. The equality is the coordinate martingale quadratic-variation
identity: its terminal variance is $x_i(1-x_i)$. By (R2), the
normalized error has operator norm at most $1/2$, so its barrier is at
most $\|F_x(Y)\|_F^2/(3c_b^2R^2)$. Combining this with the KL bound in
the Gibbs variational formula proves (R3). Finite-space limits preserve
the energy estimate for general probabilities. No negative correlation
or independent-rounding concentration theorem is asserted for the
selected laws.

\subsubsection{Arbitrary multicolor
probabilities}\label{arbitrary-multicolor-probabilities}

Let $p_{ic}\ge0$, $\sum_{c=1}^k p_{ic}=1$. Exactly one color must be
assigned to each item.

\begin{theorem}
There is a law on assignments $C_i\in[k]$ with \[
 \mathbb P(C_i=c)=p_{ic},\qquad
 \max_c\left\|\sum_i(\mathbf1_{C_i=c}-p_{ic})B_i\right\|
 \le K\sqrt{2q+\nu}.
 \tag{MC1}
\] If $\nu/R^2\le\nu_0$ and $2q/R^2\le q_0$, there is also a law
with the same marginals satisfying \[
 \max_c\left\|\sum_i(\mathbf1_{C_i=c}-p_{ic})B_i\right\|
 \le\frac R{2^{1/4}-1},
 \quad
 D\left(\mu\middle\|\bigotimes_i\operatorname{Cat}(p_i)\right)
 \le(\sqrt2+1)C\tau/R^2.
 \tag{MC2}
\] For uniform probabilities, at least
$k^n\exp\{-(\sqrt2+1)C\tau/R^2\}$ assignments satisfy the second
radius bound. 
\end{theorem}

\begin{proof}
At dyadic level $j$, the odd numerators in each item's
probability vector occur in even number. Pair them by a rule depending
only on that item. A pair $(i,c,c')$ corresponds to a block matrix
with $+B_i$ in color block $c$, $-B_i$ in block $c'$, and zeros
elsewhere. Its coefficient energy is at most $2q$. Each original item
appears at most once in any color block at this stage, so a common block
basis witnesses residual variance at most $\nu$. Sign all pairs
together and make opposite changes of size $2^{-j}$ in their two
probabilities. This preserves row sums and marginals, and a geometric sum
proves (MC1).

For entropy, a row has at most $2^j$ odd numerators, hence the lifted
stage trace is at most $2^j\tau$, regardless of the number of colors.
Use Gibbs radius $R2^{3j/4}$. The sum of norm errors is
$R\sum_{j\ge1}2^{-j/4}=R/(2^{1/4}-1)$, and the sum of entropy costs is
at most $\frac{C\tau}{R^2}\sum_{j\ge1}2^{j-3j/2}=(\sqrt2+1)C\tau/R^2$.
Uniform reference pair signs yield dynamics independent across items;
each row is a simplex-valued martingale ending at a vertex. Its terminal
reference law is the required product categorical law. Data processing
and dyadic approximation complete the proof.
\end{proof}

There is an endpoint categorical BSB inequality as well. Put
$c_m=(2^{1/4}-1)^{-1}$, $F_c=\sum_i(\mathbf1_{C_i=c}-p_{ic})B_i$,
and $\tau_p=\sum_i(1-\sum_c p_{ic}^2)\|B_i\|_F^2$. Under the
hypotheses of (MC2), \[
 \mathbb E_{(C_i)\sim\bigotimes_i\operatorname{Cat}(p_i)}
 \prod_{c=1}^k W_\beta\left(\frac{F_c}{2c_mR}\right)
 \ge\exp\left[-\frac C{R^2}
  \left(c_b\tau+\frac{\beta\tau_p}{3c_m^2}\right)\right].
 \tag{MC3}
\] Only different items are independent in this reference law; each item
still receives exactly one color. The witness law has the prescribed
color marginals and all final matrix cutoffs hold simultaneously.

For the additional energy calculation, martingale orthogonality in the
direct-sum Frobenius space gives $\mathbb E\sum_c\|F_c\|_F^2\le C\sum_j4^{-j}\mathbb E\tau(\text{stage }j)=C\sum_i(1-\sum_c p_{ic}^2)\|B_i\|_F^2$. Each odd color coordinate changes by $\pm2^{-j}$, and the trace of
its pair lift counts its coefficient energy once. Hence the final
equality is again coordinate quadratic variation, summed over colors.
Use the deterministic support in (MC2), the barrier bound at radius
$2c_mR$, and its KL bound to obtain (MC3).

Multicolor discrepancy has been studied~\citep{DS03}. The present guarantee includes arbitrary item-specific
probabilities, matrix outputs, exact marginals, and an explicit entropy
budget for the same law. The proof also permits matrices depending on
color: replace $q$ by $\max_{i,c}\|B_{ic}\|_F^2$, $\nu$ by
$\max_c\nu((B_{ic})_i)$, and $\tau$ by
$\sum_i\max_c\|B_{ic}\|_F^2$, using a separate basis per color.

\subsubsection{Exact quotas and their
boundary}\label{exact-quotas-and-their-boundary}

For disjoint groups $G_a$ partitioning the items, with integer
$k_a=\sum_{i\in G_a}x_i$, there is unbiased rounding satisfying \[
 \sum_{i\in G_a}Y_i=k_a\quad\hbox{for every }a,
 \qquad
 \left\|\sum_i(Y_i-x_i)B_i\right\|
 \le K\sqrt{4q+2\nu}.
 \tag{Q}
\] At each dyadic stage, pair odd coordinates within each group. The
pair coefficients are $B_i-B_j$, with energy at most $4q$ and
residual variance at most $2\nu$, because pairs are disjoint and
$(R_i-R_j)^2\preceq2R_i^2+2R_j^2$. The preceding rounding proof then
preserves every quota exactly. Dyadic approximation inside each
hypersimplex gives arbitrary initial points.

Arbitrary overlapping exact constraints are not covered. On a cycle with
$2m$ edges, require a perfect matching and start at $x_e=1/2$. There
are only two feasible outcomes, the alternating matchings. Let $B_e=1$
on one matching and zero on the other. Then $q=1,\nu=0$, but every
feasible outcome has rounding error $m/2$. The initial point is in the
convex hull of feasible integral points. Thus feasibility and small
coefficient energy alone cannot extend (Q) to arbitrary exact equations.
This does not disprove a theorem for a single arbitrary matroid base;
that separate problem is not resolved here. Classical matroid rounding
has additional concentration results that are also not claimed for these
BSB laws~\citep{CVZ10}.

\subsection{Online models and dynamic
discrepancy}\label{online-models-and-dynamic-discrepancy}

\subsubsection{Adaptive insertions and deletions with
recourse}\label{adaptive-insertions-and-deletions-with-recourse}

An update inserts or deletes a matrix. After each update a signing of
the active family must be output. Recourse counts changed signs of
surviving items. Let all arriving matrices belong to a real span of
dimension at most $r$, fixed over the stream. Let $N_t$ count all
insertions through time $t$, and put
$h_t=\lceil\log_2\max(1,N_t)\rceil$.

\begin{theorem}
A causal strategy, even against adaptive updates,
gives \[
 \left\|\sum_{i\ {\rm active}}s_i(t)B_i\right\|
 \le2K\sqrt{q_t+\nu_t}
 \tag{DY1}
\] with at most $(4h_t+2)r$ surviving sign changes per update. For
insertion-only streams, at most $r$ changes per update suffice. The
strategy uses a discrepancy oracle or exhaustive finite rounding. 
\end{theorem}

\begin{proof}
Use a binary tree of insertion slots. Every node stores
a fractional zero-sum assignment to its active descendants with at most
$r$ fractional coordinates. A leaf uses fractional value zero. At a
parent, concatenate its two children and alter only their at most $2r$
fractional coordinates. Linear dependence permits a move preserving zero
sum until a coordinate becomes integral. Repeat until at most $r$
remain. Integral child coordinates are frozen.

Only the leaf-to-root path is recomputed at an update. At each ancestor,
the old and new unions of child fractional sets contain at most $2r$
elements each. Hence the number of changed surviving coordinates grows
by at most $4r$ per level. At the root, round only its fractional
coordinates using (R1) in sign form. The root fractional assignment has
matrix sum zero, so its rounding error is the full discrepancy in (DY1).
The old and new root fractional sets add at most $2r$ changed signs. A
growing tree needs no known horizon: when capacity doubles, retain the
old tree as the left child of a new root, with an empty right child and
unchanged assignment.

For insertion only, maintain one such fractional assignment. Add the new
coefficient with value zero and move only the old fractional coordinates
and the new coordinate. Old integral coordinates remain fixed, so at
most $r$ old output signs change. This proves the theorem. The
underlying distributed floating-variable architecture is due to Gupta
and collaborators; the BSB rounding scale is the input used
here~\citep[Section 2]{GGKKS22}.
\end{proof}

For diagonal unit-column vectors, (DY1) is constant dynamic Komlós
discrepancy. The span parameter concerns a linear space of matrices, not
a bound on their generated algebra. An optimizing basis for $\nu_t$
may change with time, since it is used only in the fresh root rounding
on the current family.

The same tree works for $k$ colors and fixed per-item fractions
$p_{ic}$. At a product-simplex vertex under the exact matrix target
equations, at most $r(k-1)$ items are fractional. Indeed the dimension
of its active face is the sum of support sizes minus one, and exceeds
the constraint rank unless that sum is at most $r(k-1)$. Apply (MC1)
at the root. The dynamic result is \[
 \max_c\left\|\sum_i(\mathbf1_{C_i(t)=c}-p_{ic})B_i\right\|
 \le K\sqrt{2q_t+\nu_t},
 \tag{DY2}
\] with at most $(4h_t+2)r(k-1)$ recolored surviving items. Thus color
count enters recourse but not discrepancy. Insertion only needs
$r(k-1)$ changes per update.

A dimension-independent insertion-only alternative uses binary block
merges: recolor every newly merged block with (D). It gives discrepancy
at most $(1+\lfloor\log_2t\rfloor)K\sqrt{q_t+\nu_t}$, while each item
changes sign at most $\lfloor\log_2T\rfloor$ times through time $T$.
These are current assignments with recourse, not one permanent signing
for all historical prefixes.

\subsubsection{Irreversible online
signing}\label{irreversible-online-signing}

A universal constant guarantee against an adaptive adversary is false,
even when $\nu=0$ and $q$ is arbitrarily small. Given the current
two-dimensional signed sum $z_{t-1}$, present a vector $a_t$ of
length $\delta$ perpendicular to it. Either sign gives
$\|z_{t-1}+s_ta_t\|_2^2=\|z_{t-1}\|_2^2+\delta^2$. Thus every path has $\|z_T\|_\infty\ge\delta\sqrt{T/2}$. Embedding
as diagonal matrices gives $q=\delta^2,\nu=0$. This defeats randomized
strategies when the adversary sees past realized signs.

For randomized strategies against oblivious fixed input sequences,
Kulkarni--Reis--Rothvoss establish the growing $O(\sqrt{\log T})$
prefix scale and a matching lower bound~\citep{KRR23}. Those
results are benchmarks, not consequences newly proved here from BSB.
Stochastic and random-order arrivals can use uniform oblivious-input
guarantees by conditioning on the complete input sequence; this does not
allow future inputs to adapt to revealed signs. No irreversible online
improvement is established in this report.

\subsection{Prefix discrepancy and matrix
ordering}\label{prefix-discrepancy-and-matrix-ordering}

\subsubsection{Selectable order with a contracting
variance}\label{selectable-order-with-a-contracting-variance}

\begin{theorem}
If $\sum_iB_i=0$, an unsigned permutation
satisfies \[
 \max_t\left\|\sum_{j\le t}B_{\pi(j)}\right\|
 \le C'\left(\sqrt\nu+\sqrt q\,\lceil\log_2n\rceil\right).
 \tag{P1}
\] For nonzero families, at least \[
 2^n\exp\{-C\tau/[a^2(q+\nu)]\}
 \tag{P2}
\] signings admit an order with prefix bound
$O(\sqrt\nu+\sqrt q(1+\log n))$, for any sufficiently large fixed
universal $a$. The order may depend on the signing. For the zero family, every signing and order works. 
\end{theorem}

\begin{proof}
Fix a minimizing basis and pad once to $2^L$ slots
with zero matrices. Write $V_J=\sum_{i\in J}R_i^2$ and
$v_J=\|V_J\|$. Pair a node's indices and apply (D) simultaneously to
the block coefficients \[
 H_{ab}=(B_a-B_b)\oplus(R_a^2-R_b^2)/\sqrt q.
 \tag{P3}
\] The case $q=0$ is trivial. The energy is at most $8q$. For the
second block,
$q^{-1}\sum_{(a,b)}(R_a^2-R_b^2)^2\preceq2q^{-1}\sum_iR_i^4\preceq2V_J$,
because $R_i^4\preceq qR_i^2$. The first block has remainder
variance at most $2v_J$. Removing a diagonal component in any fixed
basis costs at most four times full variance, so the exhibited block
basis gives $\nu(H)\le8v_J$.

For $\kappa=\sqrt8K$, the same pair signs divide the node into
halves with $T_{J_\pm}=T_J/2\pm E_J$, where \[
 \|E_J\|\le\frac\kappa2\sqrt{v_J+q},\qquad
 v_{J_\pm}\le\frac12v_J+\frac\kappa2\sqrt{q(v_J+q)}.
 \tag{P4}
\] The elementary bound
$x/2+\kappa\sqrt{x+1}/2\le3x/4+\kappa^2/4+\kappa/2$ implies at depth
$\ell$ \[
 v_J\le(3/4)^\ell\nu+(\kappa^2+2\kappa)q,
\quad
 \|E_J\|\le\frac\kappa2
 \left[(\sqrt3/2)^\ell\sqrt\nu+(\kappa+1)\sqrt q\right].
\] Order each positive child before its negative child. If a node has
$m$ slots, its prefix error centered at $(t/m)T_J$ is the
appropriate child's centered error plus either $(2t/m)E_J$ or
$2(1-t/m)E_J$. These coefficients lie in $[0,1]$. A prefix therefore
pays at most one error per level, giving
$\frac\kappa{2-\sqrt3}\sqrt\nu+\frac{\kappa(\kappa+1)}2L\sqrt q$.
For a general unsigned family, add $\|T_{[n]}\|$. Deleting padded
zeros preserves every original absolute prefix value.

Finally BSB at radius $a\sqrt{q+\nu}$ gives the number (P2) of
signings with terminal norm below that radius. Their coefficient
energies and squared remainders are unchanged. Apply the unsigned
ordering result separately to each signed family. This proves the signed
assertion and retains the original law's entropy and terminal barrier.
It does not produce a single order common to all counted signings.
\end{proof}

\subsubsection{Prescribed order and interval
constraints}\label{prescribed-order-and-interval-constraints}

For a fixed order, let $\mathcal D$ be a full dyadic interval tree,
including singletons, after zero padding. Prescribe radii $r_I>0$ and
use the fixed residual basis above. Define $Q_*=\max_i\|B_i\|_F^2\sum_{I\ni i}r_I^{-2}$, $T_*=\sum_i\|B_i\|_F^2\sum_{I\ni i}r_I^{-2}$, and $V_*=\max_I\|\sum_{i\in I}R_i^2\|/r_I^2$. Whenever $Q_*\le q_0,V_*\le\nu_0$, direct-sum BSB gives \[
 \mathbb E_s\prod_{I\in\mathcal D}
 W_\beta\left(\frac{\sum_{i\in I}s_iB_i}{r_I}\right)
 \ge e^{-C\beta_* T_*}.
 \tag{P5}
\] Indeed use coefficients
$C_i=\bigoplus_{I\in\mathcal D}\mathbf1_{i\in I}B_i/r_I$. Their
parameters are exactly $Q_*,T_*$, with residual variance at most
$V_*$. Every prefix is a disjoint union of at most $L+1$ dyadic
intervals, so its norm is at most the sum of their radii. The Gibbs
identity gives a joint entropy plus sum-of-barriers bound, and at least
$2^n e^{-CT_*}$ full signings satisfy every dyadic constraint.
Multiple fixed orders are handled by including all their trees and
charging the full incidence energy.

With $h=1+\lceil\log_2n\rceil$ and constant radius
$r_I=a\sqrt{\nu+hq}$, the prefix norm is at most $ah\sqrt{\nu+hq}$.
In the scalar commuting case this is $O(\sqrt q\log^{3/2}n)$. It is
weaker as an existence bound than the classical
$O(\sqrt{\log d+\log n})$ unit-vector prefix bound. 

For sparse scalar rows, dyadic trees on each row's ordered nonzero
occurrences replace the global logarithm by
$h_j=1+\lceil\log_2\max(1,m_j)\rceil$, where $m_j$ is row support.
The sufficient row-radius condition becomes
$\max_i\sum_j a_{ji}^2h_j^3/R_j^2\le q_0$. This does not settle a
prefix Beck--Fiala bound depending only on column sparsity through
$\sqrt k\operatorname{polylog}k$.

\subsection{Sparse systems and simultaneous graph
constraints}\label{sparse-systems-and-simultaneous-graph-constraints}

\subsubsection{A general common-sign
budget}\label{a-general-common-sign-budget}

Let scalar constraints have coefficients $a_{\ell i}$ and tolerances
$h_\ell>0$, and let matrix constraints have coefficients $H_{ji}$
and radii $r_j>0$. Define $e_i=\sum_\ell a_{\ell i}^2/h_\ell^2+\sum_j\|H_{ji}\|_F^2/r_j^2$ and $E_*=\sum_i e_i$. If $\max_i e_i\le q_0$ and $\max_j\nu((H_{ji})_i)/r_j^2\le\nu_0$,
one common signing obeys all scalar and matrix constraints. Its joint
determinant partition is at least $e^{-C\beta_*E_*}$; at least
$2^n e^{-CE_*}$ such signings exist. The Gibbs law is unbiased and
bounds relative entropy plus the sum of normalized barriers.

The proof uses \[
 C_i=\operatorname{diag}((a_{\ell i}/h_\ell)_\ell)
       \oplus\bigoplus_j H_{ji}/r_j.
 \tag{S}
\] Different matrix blocks may use different bases, but every block
retains the same original sign for item $i$. The coefficient energies
add exactly to $e_i$, and the exhibited residual variance is their
block maximum. Determinants and cutoffs factor across the blocks. This
proves all assertions directly from (BSB).

For a scalar $k$-sparse incidence matrix with entries bounded by one,
(S) gives the classical Komlós consequence $O(\sqrt k)$, now with a
joint determinant and entropy guarantee. For matrix-valued sets $S_j$,
put $Q=\max_i\#\{j:i\in S_j\}\,\|H_i\|_F^2$ and $\kappa=\max_j\nu((\mathbf1_{i\in S_j}H_i)_i)$. The resulting simultaneous radius is $K'\sqrt{Q+\kappa}$, without a
logarithm in the number of sets. If every item belongs to at most $k$
sets, $H_i\succeq0$, $\operatorname{tr}H_i\le\varepsilon$, and
$\sum_{i\in S_j}H_i\preceq I$ for every set, the radius is
$O(\sqrt{\varepsilon+k\varepsilon^2})$: here $Q\le k\varepsilon^2$,
$\kappa\le4\varepsilon$.

\subsubsection{KS and arbitrary scalar side constraints
together}\label{ks-and-arbitrary-scalar-side-constraints-together}

Suppose $H_i\succeq0$, $\sum_iH_i=I_d$,
$\operatorname{tr}H_i\le\varepsilon\le1$, and $\|a_i\|_2\le1$. One
common signing satisfies \[
 \left\|\sum_i s_iH_i\right\|<K'\sqrt\varepsilon,
 \qquad\left\|\sum_i s_i a_i\right\|_\infty<K',
 \tag{KS+}
\] with at least
$2^n\exp\bigl[-C\bigl(\sum_i\|a_i\|_2^2+d\bigr)/(K')^2\bigr]$ signings and a joint unbiased entropy--barrier law. Indeed apply (S)
to
$C_i=(K')^{-1}\operatorname{diag}(a_i) \oplus(K'\sqrt\varepsilon)^{-1}H_i$.
Its energy is at most $2/(K')^2$, residual variance at most
$4/(K')^2$, and trace at most $(\sum_i\|a_i\|_2^2+d)/(K')^2$. Scalar
side constraints may have arbitrary row energy. This common-sign
assertion is stronger than separate existence of two unrelated signings.

\subsubsection{Graph adjacency and signed
degrees}\label{graph-adjacency-and-signed-degrees}

\begin{theorem}
Every finite simple graph with $m$ edges and
maximum degree $\Delta\ge1$ has at least $2^{m/2}$ edge signings
with \[
 \|A_s\|<K'\sqrt\Delta,
 \qquad\max_u\left|\sum_{e\ni u}s_e\right|<K'.
 \tag{GR}
\] They support an unbiased law of entropy cost at most $4Cm/(K')^2$.
The constant $K'$ is universal.
\end{theorem}

\begin{proof}
For $e=uv$, use $C_e=(K')^{-1}\operatorname{diag}(e_u+e_v)\oplus(K'\sqrt\Delta)^{-1}(E_{uv}+E_{vu})$. Its squared Frobenius norm is $(2+2/\Delta)/(K')^2\le4/(K')^2$. The
off-diagonal variance is at most $1/(K')^2$, because
$(E_{uv}+E_{vu})^2=E_{uu}+E_{vv}$ and their sum is the degree matrix.
The total trace is at most $4m/(K')^2$. Apply (S) and choose
$(K')^2\ge\max(\nu_0^{-1},4q_0^{-1},8C/\log2)$.
\end{proof}

The same signs control all disjoint cut rectangles:
$|\sum_{e\in E(S,T)}s_e|\le K'\sqrt{\Delta|S||T|}$, by testing $A_s$ on indicator vectors. A star forces the
$\sqrt\Delta$ order for every signing, while odd degrees prevent exact
signed degree zero. The theorem does not claim a sharp Ramanujan
constant. Historical graph-signing comparisons include Bilu--Linial and
the bipartite interlacing results of Marcus--Spielman--Srivastava; they
are not used as proof inputs~\citep{BL06,MSS15a}.

\subsubsection{Spectral edge bisection with additive degree
control}\label{spectral-edge-bisection-with-additive-degree-control}

For a nonempty positively weighted loopless graph with Laplacian $L$,
let $b_e=e_u-e_v$, $\ell_e=w_e b_e^{\mathsf T}L^\dagger b_e$,
$\varepsilon=\max_e\ell_e$, and $w_{\max}=\max_e w_e$. On the range
of $L$, set $H_e=L^{\dagger/2}w_e b_e b_e^{\mathsf T}L^{\dagger/2}$ and $a_e=\frac{w_e}{\sqrt2 w_{\max}}(e_u+e_v)$. These satisfy $\sum_eH_e=I$,
$\operatorname{tr}H_e=\ell_e\le\varepsilon\le1$, and
$\|a_e\|_2\le1$. Apply (KS+) and take the positive edges. There is
consequently an unbiased half-selection $F$ supported on $|x^{\mathsf T}(2L_F-L)x|<K'\sqrt\varepsilon\,x^{\mathsf T}Lx$ for $x^{\mathsf T}Lx>0$, and $|2d_F(u)-d_G(u)|<\sqrt2K'w_{\max}$ for every $u$.
The expression vanishes on the kernel of $L$. Counts and the joint
entropy budget are controlled by \[
 \frac C{(K')^2}\left[
 \sum_e(w_e/w_{\max})^2+\varepsilon^{-1}\sum_e\ell_e^2\right]
 \le\frac C{(K')^2}\left[
 \sum_e(w_e/w_{\max})^2+\operatorname{rank}L\right].
\] When $K'\sqrt\varepsilon<1$, both parts are spectral approximations
to half the graph. This is selection of original edges with their
original weights and additive degree error, rather than exact degree
preservation or a reweighted sparsifier.

\subsection{Interpretation and remaining
limits}\label{interpretation-and-remaining-limits}

The main additional BSB statements established here are centered
Bernoulli and categorical determinant partitions with unbiased
entropy: barrier witnesses, the variance-contracting ordering
construction, and the fixed-degree multipartite vector partition family. Dynamic and graph statements then follow by explicit reductions that preserve their information constraints and common signs. All positive
statements retain the same dependencies as the master
theorem.

The analysis does not prove constant prescribed-order prefixes, a prefix
Beck--Fiala bound depending only on column sparsity, an irreversible
online improvement over the cited bounds, general centered quadratic
discrepancy, or arbitrary matroid-constrained rounding. Some broader
formulations are false by the counterexamples given above. The existence
of these limits is part of the result; none is replaced by an untested
comparison principle.

The constructions are existential unless an explicit computation is
described. The dynamic root can be rounded by finite enumeration, but no
efficient sampler for the BSB Gibbs law is supplied. The checks
accompanying the report verify finite-dimensional lift identities,
reference rounding laws, variance recurrences, and dynamic invariants.
They are supplementary checks of the reductions, not numerical proofs of
the master theorem or independent validations of its two source
preprints. Historical comparisons identify the relevant models; an
exhaustive priority claim is not made.

\section{Interpolating between  Gaussian and Boolean Small-Ball }
\label{sec:families}

The exact cutoff weight is stable under weak limits, by
Lemma~\ref{lem:determinant}. This makes it possible to derive a full
interpolation family from BSB without differentiating its partition
function.

\subsection{Gaussian--Boolean interpolation}

\begin{theorem}[Heterogeneous Gaussian--Boolean family]
\label{thm:gaussian-boolean}
Let $g_i$ be independent standard Gaussians, let $s_i$ be
independent uniform signs, and assume the two families are independent.
For deterministic $\theta_i\in[0,1]$, put $X_i=\sqrt{\theta_i}\,g_i+\sqrt{1-\theta_i}\,s_i$ and $q_{\mathrm{res}}=\max_i(1-\theta_i)w_i$.
If $v\le v_0$ and $q_{\mathrm{res}}\le q_0$, then
\begin{equation}\label{eq:gb}
 \E W_\beta\!\left(\nlsum_iX_iB_i\right)
 \ge e^{-C\bstar\tau}\qquad(\beta>0).
\end{equation}
\end{theorem}

\begin{proof}
For a positive integer $m$, use fresh independent signs $r_{ij}$
and define
\[
 S^{(m)}
 =\sum_i\sqrt{1-\theta_i}\,s_iB_i
  +\sum_{i=1}^n\sum_{j=1}^m\sqrt{\theta_i/m}\,r_{ij}B_i.
\]
The expanded Boolean family has exactly the same variance matrix and
total trace:
\begin{equation}\label{eq:gb-parameters}
 V^{(m)}=V,\qquad \tau^{(m)}=\tau,\qquad
 q^{(m)}=\max_i\{(1-\theta_i)w_i,\theta_iw_i/m\}.
\end{equation}
For all sufficiently large $m$, the last quantity is at most $q_0$.
Apply Theorem~\ref{thm:main}. The joint central limit theorem gives $S^{(m)}\Rightarrow\sum_i\bigl(\sqrt{1-\theta_i}\,s_i+\sqrt{\theta_i}\,g_i\bigr)B_i$.
Since $W_\beta$ is bounded and continuous, its expectations converge,
and the bound passes to the limit.
The constants are exactly those of Theorem~\ref{thm:main}.
\end{proof}

For a common parameter $\theta$, the hypothesis is
$(1-\theta)q\le q_0$. At $\theta=0$ one obtains BSB; at
$\theta=1$ no coefficient-energy restriction remains: $v\le v_0$ implies $\E_gW_\beta(S_g)\ge e^{-C\bstar\tau}$.
The same constants work at every point of the path.

\begin{proposition}[Entropy ordering of the path]\label{prop:entropy-path}
The law of $X_\theta=\sqrt{\theta}\,g+\sqrt{1-\theta}\,s$ has
variance one. Its differential entropy is nondecreasing on
$0<\theta\le1$, with maximum at the Gaussian endpoint.
As $\theta\downarrow0$, the laws converge weakly to the uniform
sign law.
\end{proposition}

\begin{proof}
For $0<\theta_1<\theta_2\le1$, let $a^2=(1-\theta_2)/(1-\theta_1)$.
With an independent standard Gaussian $g'$,
$X_{\theta_2}\stackrel{\dd}{=}aX_{\theta_1}+\sqrt{1-a^2}\,g'$.
This Gaussian channel preserves the standard Gaussian law $\gamma$.
Its relative-entropy contraction follows from Jensen's inequality:
if $f=\dd P/\dd\gamma$, the output density relative to $\gamma$
is a conditional expectation of $f$ under the stationary channel,
and $x\mapsto x\log x$ is convex. Hence
$D(\Law(X_{\theta_2})\|\gamma)\le D(\Law(X_{\theta_1})\|\gamma)$.
For these mean-zero variance-one laws,
$D(\Law(X_\theta)\|\gamma)=\tfrac12\log(2\pi e)-h(X_\theta)$,
which proves the entropy assertion. Their positive-$\theta$
densities are two-component Gaussian mixtures, so these entropies
are finite. Weak convergence at zero follows directly from the
definition. No differential entropy is assigned to the discrete
endpoint.
\end{proof}

\Needspace{15\baselineskip}
\subsection{Signed magnitudes and power laws}

\begin{theorem}[Conditional magnitude principle]\label{thm:magnitudes}
Let $R_i\ge0$ be jointly distributed random magnitudes, independent of
independent uniform signs $s_i$. Suppose, almost surely,
\begin{equation}\label{eq:magnitude-hyp}
 \left\|\nlsum_iR_i^2B_i^2\right\|\le v_0,\qquad
 \max_i R_i^2w_i\le q_0.
\end{equation}
Then
\begin{equation}\label{eq:magnitudes}
 \E W_\beta\!\left(\nlsum_i s_iR_iB_i\right)
 \ge
 \exp\left\{-C\bstar\,\E\nlsum_iR_i^2w_i\right\}.
\end{equation}
The magnitudes need not be independent.
\end{theorem}

\begin{proof}
Condition on $R=(R_1,\ldots,R_n)$. The coefficients $R_iB_i$
satisfy the hypotheses of Theorem~\ref{thm:main}, so the conditional partition
is at least $\exp\{-C\bstar\sum_iR_i^2w_i\}$. Average this inequality
and apply Jensen to the exponential. The expectation in
\eqref{eq:magnitudes} is finite by \eqref{eq:magnitude-hyp};
zero coefficients may first be removed.
\end{proof}

\begin{corollary}[Power-law family]\label{cor:power}
For $\lambda_i>0$, let $R_i$ have distribution function
$\Prob(R_i\le r)=r^{\lambda_i}$, $0\le r\le1$, and be independent
of independent signs $s_i$. Under $v\le v_0,q\le q_0$,
\begin{equation}\label{eq:power}
 \E W_\beta\!\left(\nlsum_i s_iR_iB_i\right)
 \ge\exp\left\{-C\bstar
       \nlsum_i\frac{\lambda_i}{\lambda_i+2}w_i\right\}.
\end{equation}
For the variance-one variables $\widehat X_i=\sqrt{(\lambda_i+2)/\lambda_i}\,s_iR_i$, $\lambda_i\ge1$,
one has
\begin{equation}\label{eq:standard-power}
 v\le v_0/3,\quad q\le q_0/3
 \quad\Longrightarrow\quad
 \E W_\beta\!\left(\nlsum_i\widehat X_iB_i\right)
 \ge e^{-C\bstar\tau}.
\end{equation}
\end{corollary}

\begin{proof}
The magnitudes are at most one, and integration gives
$\E R_i^2=\lambda_i/(\lambda_i+2)$.
Apply Theorem~\ref{thm:magnitudes}. For the standardized variables, the
squared magnitudes are at most three and have expectation one, which
gives \eqref{eq:standard-power}.
\end{proof}

When all $\lambda_i=\lambda$, the coefficient law has density
$\frac{\lambda}{2}|x|^{\lambda-1}\mathbf1_{(-1,1)}(x)$.
The estimates are uniform as $\lambda\to\infty$, when the law
converges to the uniform sign law. This limiting statement is a
consistency consequence of BSB, not an assumption in its proof.

\subsection{A combined family}

\begin{theorem}[Gaussian terms with signed magnitudes]\label{thm:combined}
Let $a_i\ge0$ be deterministic. Let $g_i,s_i$ be independent
standard Gaussian and uniform sign variables, independent of the
joint magnitude vector $R$. If, almost surely,
\begin{equation}\label{eq:combined-hyp}
 \left\|\nlsum_i(a_i^2+R_i^2)B_i^2\right\|\le v_0,\qquad
 \max_iR_i^2w_i\le q_0,
\end{equation}
then
\begin{equation}\label{eq:combined}
 \E W_\beta\!\left(\nlsum_i(a_ig_i+s_iR_i)B_i\right)
 \ge\exp\left\{-C\bstar\nlsum_i(a_i^2+\E R_i^2)w_i\right\}.
\end{equation}
\end{theorem}

\begin{proof}
Condition on $R$. Set $c_i=(a_i^2+R_i^2)^{1/2}$, omitting any
zero $c_i$, and apply Theorem~\ref{thm:gaussian-boolean} to coefficients
$c_iB_i$ with $\theta_i=a_i^2/c_i^2$.
Its variance condition is the first inequality in
\eqref{eq:combined-hyp}, while its residual coefficient energy is
exactly $R_i^2w_i$. The conditional total trace is
$\sum_i(a_i^2+R_i^2)w_i$. Average the resulting bound and apply
Jensen's inequality.
\end{proof}

The coefficient-energy condition in Theorem~\ref{thm:combined} is imposed only
on the signed magnitudes. Gaussian terms can be divided into
arbitrarily many independent smaller coefficients, as in
\eqref{eq:gb-parameters}.

\section{Constructive coloring for Kadison--Singer}
\label{app:ks}
\darkred{NOTE: We decided to retain the section below for completeness, and for noting an observation that might be of interest to some readers; it has not yet received a rigorous round of care and editing. We will update it in the next version of the paper.}

\subsection{Upper bounds for the constructive method}
\citet{EJ26} give a deterministic polynomial-time
signing algorithm for rank-1 KS problems. We record below an easy extension of their work to the general rank case.

The following trace-map modification extends
their construction to rational PSD Hermitian matrices of arbitrary rank:
\begin{equation}
\label{eq:ks-const}
 \left\|\nlsum_{i=1}^n s_i A_i\right\|
 \le16\left\|\nlsum_{i=1}^n\operatorname{tr}(A_i)A_i\right\|^{1/2}.
\end{equation}
In particular, $\sum_i A_i=I$ and $\operatorname{tr}(A_i)\le\varepsilon$
give discrepancy at most $16\sqrt\varepsilon$, with one sign per matrix.

\emph{Extension and constants.}
The all-zero case is trivial.  Otherwise discard zero matrices, put
$T=\sum_i\operatorname{tr}(A_i)$, $M_i=A_i/T$, and
$W=\sum_i\operatorname{tr}(M_i)M_i$.  Choose rational parameters
\[
 \|W\|\le\nu\le\min\{1,1.001\|W\|\},\quad
 c=60,\quad \rho=\nu/d,\quad
 \frac{\sqrt\nu}{200n}\le\lambda\le\frac{\sqrt\nu}{100n}.
\]
For $x\in[-1,1]^n$, use $S=\sum_i x_iM_i$,
$\psi_i=(1-x_i^2)^{1/3}$, $\Phi=\sum_i\psi_i$, and replace the
rank-one reservoir by the positive self-adjoint map
\[
 \mathcal E_x(Z)=c\sum_i\psi_i\operatorname{tr}(M_iZ)M_i.
\]
Run their descent algorithm on $\Psi=R+\lambda\Phi$, where
\[
 R(x)=\min_{t,X,Y\succ0}\left\{
 t+\rho\operatorname{tr}(X+Y):
 \begin{array}{l}
 X^{-1}+S+\mathcal E_x(Y)\preceq tI,\\
 Y^{-1}-S+\mathcal E_x(X)\preceq tI
 \end{array}\right\}.
\]
This is an SDP by Schur complements.  For $M,N\succeq0$, $X\succ0$,
and $D=D^*$, the needed replacements for rank-one identities are
\[
 \begin{gathered}
 MXM\preceq\operatorname{tr}(MX)M,\qquad
 0\le\operatorname{tr}(MXNX)
 \le\operatorname{tr}(MX)\operatorname{tr}(NX),\\
 |\operatorname{tr}(MD)|^2
 \le\operatorname{tr}(MX)\operatorname{tr}(MDX^{-1}D).
 \end{gathered}
\]
Thus, in~\cite[Section~4.2]{EJ26}, replace $v_i^*Zv_i$ by
$\operatorname{tr}(M_iZ)$ and $|v_i^*Xv_j|^2$ by
$\operatorname{tr}(M_iXM_jX)$.  The response and dual equations remain
exact; the interaction Gram matrices remain PSD and entrywise nonnegative;
the diagonal energy identities become upper bounds in the required direction.

For completeness, the following verifies the numerical constants, using
the rescaled notation of that section.  With $m$ active coordinates,
truncate the two weighted interaction matrices at singular value
$\theta=1/3$ and use covariance parameter $\kappa=3$.
Each squared Frobenius norm is at most $m/c$, giving
\[
 \operatorname{tr}\Gamma\ge
 \left(1-\frac{2}{c\theta^2}-\frac{2}{\kappa}\right)m=\frac{m}{30},
 \qquad
 \mathbb E E_h\le\frac{27}{4}
 \sum_i(\widetilde a_i z_i+\widetilde b_i w_i)\omega_i.
\]
Young's inequality $2\sqrt{EB}\le(3/2)E+(2/3)B$ improves the curvature
estimate to $\nabla^2R[h,h]\le(7/2)E_h-(2c/3)J_h$, where
$J_h=\sum_i(-\psi_i'')h_i^2(b_ip_i+a_iq_i)$.
Also $(-\psi'')/(\alpha\beta)\ge3/5$: by symmetry take $x\ge0$;
the light-coordinate inequalities give $\alpha\le1$ and
$\beta\le(3-x)/(3(1-x))$, whence
\[
 \frac{-\psi''}{\alpha\beta}
 \ge\frac{2(3+x^2)}{3(3+2x-x^2)(1-x^2)^{2/3}}
 \ge\frac{2(3+x^2)}{3(3+2x-3x^2)}>\frac35.
\]
Here $(1-x^2)^{2/3}\le1-2x^2/3$ and $37x^2-18x+3>0$.
Set $\Sigma=\mathbb E hh^\top$,
$\mu_i=-(189/16)(\widetilde a_i-\widetilde b_i)\omega_i/T_i$, and
$\mathcal Df=\nabla f\cdot\mu+\tfrac12\operatorname{tr}(\nabla^2f\Sigma)$.
The drift cancellation now gives
\[
 \mathcal DR\le-\frac{3}{16}
 \sum_i(\widetilde a_iw_i+\widetilde b_i z_i)\omega_i\le0,
 \qquad
 \mathcal D\Phi\le-\frac{17}{320}\sum_i\omega_i.
\]
The latter uses $|\psi_i'\mu_i|\le(63/160)(-\psi_i'')\omega_i/T_i^2$
and $(-\psi_i'')/T_i^2\ge1/2$.
Consequently the same endpoint/gradient/negative-curvature algorithm
reaches a vertex without increasing $\Psi$; no randomization is needed.
Its movement and precision constants are enlarged to accommodate
$\operatorname{tr}\Gamma\ge m/30$ and the displayed drift margin.
Polynomial bit complexity is preserved: $\nu\ge1/(nd)$,
$\rho^{-1}\le nd^2$, and the positive-map KKT conditioning argument
applies unchanged.  Snap coordinates within $\lambda^2/(2\nu)$ of
an endpoint and evaluate the SDP and its derivatives to inverse-polynomial
accuracy, as in the finite-step descent.
Finally, $\mathcal E_x(I)\preceq c\nu I$ implies
\[
 \|S(s)\|\le\Psi(0)
 \le\bigl(2\sqrt{62}+0.01\bigr)\sqrt\nu
 \le15.766\sqrt{\|W\|}<16\sqrt{\|W\|}.
\]
Rescaling proves the claim. 

\vskip12pt
\noindent\textbf{Note:} More involved arguments yield a constant $< 4$ in~\eqref{eq:ks-const}, but since optimizing the approximation constant is orthogonal to this paper we omit that discussion. We do, however, note some lower bounds on the approximation factor in the section below.


\subsection{Lower bounds for Kadison--Singer signing}
For a family $A=(A_1,\ldots,A_n)$ of positive semidefinite matrices
satisfying $\sum_i A_i=I_d$, write
\[
 \Delta(A):=\min_{s\in\{-1,1\}^n}
       \left\|\nlsum_i s_iA_i\right\|,
 \qquad
 \varepsilon(A):=\max_i\operatorname{tr}(A_i).
\]
Here and below, $\|\cdot\|$ denotes the operator norm.
The coefficient in an absolute bound
$\Delta(A)\le C\sqrt{\varepsilon(A)}$ is distinct from an approximation
ratio relative to $\Delta(A)$.

\begin{proposition}[An unconditional lower bound]
\label{prop:ks-constant-lower-bound}
For every $0<C<2$ and $\varepsilon_0>0$, there are rational, real,
rank-one matrices $A_i\succeq0$ such that
\[
 \nlsum_i A_i=I_d,\qquad
 0<\varepsilon(A)<\varepsilon_0,\qquad
 \Delta(A)>C\sqrt{\varepsilon(A)}.
\]
Consequently, every universal Kadison--Singer signing coefficient is at least $2$.
\end{proposition}

\begin{proof}
Fix $0<\theta<1/2$, let $n$ tend to infinity through even integers,
and choose $d/n\to\theta$.  Let $V$ consist of the first $d$ rows of
a Haar-distributed $Q\in SO(n)$, and set $A_i=v_iv_i^\top$, where $v_i$
are the columns of $V$.  Then $\sum_i A_i=VV^\top=I_d$.
For $D_0=\operatorname{diag}(I_{n/2},-I_{n/2})$, the empirical
spectral measure of $S_0=VD_0V^\top$ converges in probability to
the real Jacobi limiting law
\[
 \rho_\theta(x)=
 \frac{\sqrt{b_\theta^2-x^2}}{2\pi\theta(1-x^2)}
       \mathbf{1}_{\{|x|<b_\theta\}},
 \qquad b_\theta=2\sqrt{\theta(1-\theta)};
\]
see the random-subspace interpretation and Wachter law
in~\cite[Sections~2.1.2 and~2.2]{Johnstone2008}.

Fix $0<t<b_\theta$ and put
$F(Q)=d^{-1}\operatorname{tr}(S_0-tI)_+$, with $(\cdot)_+$
denoting the positive part.  Since the spectra lie in $[-1,1]$,
\[
 \mathbb{E}F\longrightarrow
 m_{\theta,t}:=\int_t^{b_\theta}(x-t)\rho_\theta(x)\,dx>0.
\]
The Hoffman--Wielandt inequality and
$\|VD_0V^\top-WD_0W^\top\|_F\le2\|V-W\|_F$
show that $F$ is $2/\sqrt d$-Lipschitz in the Frobenius metric.
Concentration on $SO(n)$~\cite[Proposition~2.2]{MeckesMeckes2013}
therefore gives, for all sufficiently large $n$,
\[
 \mathbb{P}\{\lambda_{\max}(S_0)\le t\}
 =\mathbb{P}\{F=0\}
 \le \exp(-c_{\theta,t}nd),
 \qquad c_{\theta,t}>0.
\]
For any sign vector $s$, negate $s$ if necessary so that at least
$n/2$ entries are positive.  There is then a balanced diagonal sign
matrix $D'\preceq\operatorname{diag}(s)$, whose compression has the
same law as $S_0$.  Thus
$\|V\operatorname{diag}(s)V^\top\|\le t$ implies
$\lambda_{\max}(VD'V^\top)\le t$.  A union bound over all signings yields
\[
 \mathbb{P}\{\Delta(A)\le t\}
 \le 2^n\exp(-c_{\theta,t}nd)=o(1).
\]
Also, $Q\mapsto\|v_i\|^2$ is $2$-Lipschitz and has mean $d/n$.
The same concentration inequality and a union bound give
$\varepsilon(A)\le\theta+h$ with probability tending to one,
for every fixed $h>0$.

Given $C<2$ and $\varepsilon_0>0$, first choose $\theta>0$ small
enough that $2\sqrt{1-\theta}>C$ and $\theta<\varepsilon_0$,
and then choose $h>0$ and $t$ with
\[
 \theta+h<\varepsilon_0,\qquad
 C\sqrt{\theta+h}<t<b_\theta.
\]
For sufficiently large $n$, both required inequalities hold
simultaneously for some $Q$.  Finally, rational orthogonal matrices
are dense in $SO(n)$, by the Cayley transform of rational
skew-symmetric matrices.  Approximating this $Q$ preserves the strict
inequalities and the exact identity $VV^\top=I_d$, and gives rational
rank-one summands.
\end{proof}

\begin{proposition}[An NP-hard gap at the square-root scale]
\label{prop:ks-hardness-gap}
There is an absolute constant $c_0>0$ such that, for every fixed
integer $q\ge2$ and $\varepsilon_q=3/(2q^2)$, it is NP-hard to
distinguish
\[
 \Delta(A)=0
 \qquad\text{from}\qquad
 \Delta(A)\ge c_0\sqrt{\varepsilon_q},
\]
on rational PSD instances satisfying
$\sum_iA_i=I$, $\operatorname{rank}(A_i)\le3$, and
$\varepsilon(A)\le\varepsilon_q$.
\end{proposition}

\begin{proof}
We use the bounded-occurrence set-splitting gap
of Spielman and Zhang~\cite[Lemma~2]{SpielmanZhang2022}:
for an absolute $\gamma>0$, it is NP-hard to distinguish satisfiable
instances from instances in which every signing violates at least a
$\gamma$ fraction of the constraints.  Each constraint has four
distinct variables, is satisfied when their signs sum to zero, and
each variable occurs at most three times.

Put $k=q^2$ and $L=\binom{k}{2}$.  Take $L$ disjoint copies of an
instance with $m$ constraints, and index the resulting $mL$
constraints by $S_j$.  For each variable define
\[
 D_i=\frac14\operatorname{diag}
       \bigl(\mathbf{1}_{\{i\in S_j\}}\bigr)_{j=1}^{mL}.
\]
Then $\sum_iD_i=I_{mL}$ and $\operatorname{rank}(D_i)\le3$.
In the satisfiable case some signed sum is zero.  In the other case,
every signed sum $D=\sum_i s_iD_i$ has at least $\gamma mL$
diagonal entries of absolute value $\ge 1/2$.

We apply the complete-graph compression underlying
\cite[Lemma~10]{SpielmanZhang2022} directly to these matrices.
Let $B\in\mathbb{R}^{k\times L}$ be an oriented incidence matrix of
$K_k$, and define the rational matrices
\[
 \Pi=
 \left[-\frac{\mathbf{1}_{k-1}}q\ \middle|\
 I_{k-1}-\frac{J_{k-1}}{q(q+1)}\right],
 \qquad
 G=\frac{\Pi B}{q}.
\]
Here $J_r$ is the $r\times r$ all-ones matrix.
Direct calculation gives $\Pi\mathbf{1}_k=0$,
$\Pi\Pi^\top=I_{k-1}$, $GG^\top=I_{k-1}$, and
$\|Ge_e\|^2=2/k$ for every edge coordinate $e$.
Partition the $mL$ coordinates into $m$ blocks of size $L$, put
$F=\bigoplus_{a=1}^m G$, and set $A_i=FD_iF^\top$.  Consequently,
\[
 \nlsum_i A_i=I_{m(k-1)},\qquad
 \operatorname{rank}(A_i)\le3,\qquad
 \operatorname{tr}(A_i)\le\frac{3}{2k}.
\]

For diagonal $E\in\mathbb{R}^{L\times L}$, the matrix
$BEB^\top$ annihilates $\mathbf{1}_k$.  Its off-diagonal
entries give $\|BEB^\top\|_F^2\ge2\|E\|_F^2$.  Since $\Pi$ is
an isometry on $\mathbf{1}_k^\perp$,
\[
 \|GEG^\top\|^2
 \ge\frac{\|GEG^\top\|_F^2}{k-1}
 =\frac{\|BEB^\top\|_F^2}{k^2(k-1)}
 \ge\frac{2\|E\|_F^2}{k^2(k-1)}.
\]
In the unsatisfiable case, some block $E$ of $D$ has at least
$\gamma L$ entries of absolute value $\ge 1/2$.  Hence
\[
 \left\|\nlsum_i s_iA_i\right\|^2
 \ge\frac{2(\gamma L/4)}{k^2(k-1)}
 =\frac{\gamma}{4k}
 =\frac{\gamma}{6}\,\varepsilon_q.
\]
The satisfiable case still has a zero signed sum.  This proves the
claim with $c_0=\sqrt{\gamma/6}$.  The construction is rational and
has size and bit complexity polynomial in the source instance size
and $k$; it retains one sign for each matrix $A_i$.
\end{proof}

In particular, unless $\mathrm{P}=\mathrm{NP}$, no polynomial-time
algorithm can always return a signing with discrepancy at most
$\Delta(A)+c\sqrt{\varepsilon(A)}$ for a fixed $0<c<c_0$.
The analogous bounded-error randomized guarantee would imply
$\mathrm{NP}\subseteq\mathrm{BPP}$.
Proposition~\ref{prop:ks-constant-lower-bound} is an obstruction
to existence, whereas Proposition~\ref{prop:ks-hardness-gap}
concerns optimization relative to the best signing.  Neither
establishes a separation between the best existential and
polynomial-time absolute signing coefficients.

\end{document}